\documentclass[11pt]{article}

\usepackage[dvipsnames]{xcolor}
\usepackage{amsthm}
\usepackage{amscd}
\usepackage{amsmath}
\usepackage{amssymb}
\usepackage{amsfonts}
\usepackage{bm}
\usepackage{bbm}
\usepackage{booktabs}
\usepackage{color, colortbl}
\usepackage{dsfont} 
\usepackage{enumerate}
\usepackage{graphicx}
\usepackage{mathabx}
\usepackage{multirow}
\usepackage{mathrsfs}
\usepackage{mdframed}
\usepackage{mathbbol}
\usepackage{mathtools}
\usepackage{stmaryrd}

\usepackage{natbib}
\usepackage{hyperref}

\usepackage{setspace}
\usepackage{url}

\usepackage{geometry}
\usepackage[shortlabels]{enumitem}
\usepackage{caption}
\usepackage{subcaption}

\usepackage{tikz}

\newcommand{\R}{\mathbb{R}}
\newcommand{\E}{\mathbb{E}}
\newcommand{\Prob}{\mathbb{P}}

\newcommand{\W}{\mathbf{W}}

\newcommand{\rW}{\overbar{W}}
\newcommand{\rX}{\overbar{X}}
\newcommand{\rZ}{\overbar{Z}}
\newcommand{\rY}{\overbar{Y}}

\newcommand{\rw}{\overbar{w}}

\newcommand{\fF}{\mathscr{F}}

\newcommand{\FF}{\mathbb{F}}
\newcommand{\nFF}{\tilde{\mathbb{F}}}
\newcommand{\fG}{\mathscr{G}}

\newcommand{\FG}{\mathbb{G}}

\newcommand{\dx}[1]{\ \text{\textup{d}} #1}
\newcommand{\indicator}[1]{\mathds{1}\!\left\{ #1 \right\}}
\newcommand{\indep}{\rotatebox[origin=c]{90}{$\models$}}

\newcommand{\overbar}[1]{\mkern 1.5mu\overline{\mkern-1.5mu#1\mkern-1.5mu}\mkern 1.5mu}

\newcommand{\Qrob}{\mathbb{Q}}

\newcommand{\cadlag}{c\`adl\`ag }
\newcommand{\Ito}{It\^{o} }
\newcommand{\fH}{\mathscr{H}}

\newcommand{\FH}{\mathbb{H}}

\newcommand{\bbzero}{\bar{\boldsymbol{0}}}
\newcommand{\bbone}{\bar{\boldsymbol{1}}}
\newcommand{\bmu}{\boldsymbol{\mu}}
\newcommand{\tbmu}{\tilde{\boldsymbol{\mu}}}

\DeclarePairedDelimiter\floor{\lfloor}{\rfloor}

\theoremstyle{plain}
\newtheorem{thm}{Theorem}
\newtheorem{prop}{Proposition}
\newtheorem{cor}{Corollary}

\theoremstyle{definition}
\newtheorem{defn}{Definition}
\newtheorem{assum}{Assumption}

\theoremstyle{remark}
\newtheorem{remark}{Remark}

\newenvironment{assumbis}[1]
  {\renewcommand{\theassum}{\ref{#1}$'$}%
   \addtocounter{assum}{-1}%
   \begin{assum}}
  {\end{assum}}

\allowdisplaybreaks

\newcommand{\blind}{0}

\title{Causal identification for continuous-time stochastic processes}

\if1\blind
{
  \author{[Blinded for review]}
} \fi

\if0\blind
{
\author{Jinghao Sun$^{1}$ and Forrest W. Crawford$^{1,2,3,4}$ \\[1em]
1. Department of Biostatistics, Yale School of Public Health \\
2. Department of Statistics \& Data Science, Yale University \\
3. Department of Ecology \& Evolutionary Biology, Yale University \\
4. Yale School of Management \\
} 
} \fi

\begin{document}

\maketitle
\begin{abstract}
\noindent Many real-world processes are trajectories that may be regarded as continuous-time ``functional data''. Examples include patients' biomarker concentrations, environmental pollutant levels, and prices of stocks. 
Corresponding advances in data collection have yielded near continuous-time measurements, from e.g. physiological monitors,  wearable digital devices, and environmental sensors. Statistical methodology for estimating the causal effect of a time-varying treatment, measured discretely in time, is well developed. But discrete-time methods like the g-formula, structural nested models, and marginal structural models do not generalize easily to continuous time, due to the entanglement of uncountably infinite variables.  Moreover, researchers have shown that the choice of discretization time scale can seriously affect the quality of causal inferences about the effects of an intervention.  In this paper, we establish causal identification results for continuous-time treatment-outcome relationships for general \cadlag stochastic processes under continuous-time confounding, through orthogonalization and weighting. We use three concrete running examples to demonstrate the plausibility of our identification assumptions, as well as their connections to the discrete-time g methods literature. \\[1em] 


\end{abstract}



\newpage

\section{Introduction}


Many variables of interest are trajectories that can be considered as continuous-time stochastic processes taking values in finite, countable or uncountable state spaces.
Longitudinal data have long been collected sparsely over time as proxies for continuous trajectories, e.g. the Framingham Heart Study \citep{mahmood2014framingham}, the Nurses' Health Study \citep{colditz1997nurses}, and the Survey of Health, Ageing, and Retirement in Europe (SHARE) \citep{borsch2013data}.
Recent technology advances in data collection have yielded near continuous-time dense measurements, i.e. functional data. For example, heart rates from physiological monitors, physical activities from wearable trackers, PM2.5 levels from air quality sensors, stock prices from online trading platforms, and brain images from functional magnetic resonance imaging (fMRI) techniques.
A wide variety of empirical studies have focused on the relationship between two trajectories. Examples include trajectories of close interpersonal contact on COVID-19 incidence \citep{crawford2022impact}, climate changes on agriculture \citep{blanc2020use}, trajectories of low-density lipoprotein cholesterol levels on cardiovascular disease risks \citep{domanski2020time}, and socioeconomic status trajectories on cognitive functions \citep{lyu2016socioeconomic}.
	
A central question in many scientific inquiries and policy evaluations is the causal effect of one trajectory (i.e. the \emph{treatment} or \emph{intervention} or \emph{exposure}) on another trajectory (i.e. the \emph{outcome}). Randomized experiments are often seen as the gold standard design for estimating unbiased causal effects. However, such experiments can be unethical, too difficult, or too expensive to implement in some scenarios. Instead, researchers often must resort to observational studies, where time-varying confounding and treatment-confounder (outcome) feedback can occur over time. Na\"ive adjustments for confounders in this setting can result in biased estimates of causal effects since the confounders may also be mediators of the effects of previous treatments on future outcomes \citep[Chap.~19,~20]{hernan2020july}.

Most of the existing literature on causal inference for continuous trajectories has focused on the case of longitudinal data. For example, \emph{g methods} \citep[Chap.~21]{hernan2020july}, including the g-formula \citep{gill2001causal}, marginal structural models and inverse probability-of-treatment weighting (IPTW) methods \citep{robins2000marginal}, and structural nested models and g-estimation methods \citep{robins1989analysis}, work with longitudinal data where each subject is measured at the same time points (i.e. \emph{regular} longitudinal data).
These methods can identify the causal effects of a sequence of treatments under assumptions such as sequential exchangeability, sequential consistency, and sequential positivity. 
These methods have been widely applied in empirical research in public health \citep[e.g.][]{thompson2021prevention, vangen2018hypothetical, vock2017survival, saul2019downstream}.

However, when the data generating process (DGP) is continuous in time, issues arise for g-methods with discretely measured longitudinal data. For example, as \citet{sun2022role} point out, the discrete-time potential outcome notation has implicitly assumed \emph{no multiple versions of treatment} or \emph{treatment variation irrelevance} \citep{vanderweele2009concerning}, but this is rarely true: many continuous-time trajectories might coincide with a single observed sequence of discretely sampled treatments. These multiple distinct trajectories passing through the observe data comprise different versions of treatment, rendering some potential outcome notation ill-defined. In addition, the treatment assignment might depend on the unobserved values of the past treatment and covariate processes, causing violations of \emph{unconfoundedness}.

Several causal estimands and identification strategies for continuous-time treatments and outcomes have been proposed. Most prior literature focuses on piecewise-constant treatment processes, e.g. time-to-event treatments, counting process treatments, and (marked) point process treatments. Among them, \citet{yang2021semiparametric2}, \citet{zhang2011causal}, and \citet{lok2008statistical} generalize g-estimation; \citet{roysland2022graphical}, \citet{hu2021joint}, \citet{saarela2016flexible}, \citet{roysland2011martingale}, and \citet{johnson2005semiparametric} generalize IPTW methods; and \citet{rytgaard2022continuous} generalize the g-formula. 

However, many exposures in empirical studies change continuously and take values in an uncountable set, e.g. air pollutant concentrations and blood pressures. In this paper, we develop novel causal identification strategies for general \cadlag trajectories of treatments, covariates, and outcomes under the potential outcome framework. Examples include discrete-time DGPs, piecewise-constant processes, continuous processes (e.g. diffusions), and continuous processes with jumps (e.g. diffusions with jumps). Therefore, the proposed strategies provide a unifying solution to a large class of continuous-time causal questions. 

In the following, we define causal estimands based on continuous-time potential outcomes, propose causal identification assumptions that generalize those employed in discrete-time g-methods, and present identification strategies mainly based on a continuous-time \emph{unconfoundedness} condition. Our approach identifies causal parameters based on moment conditions. It takes the form of \emph{orthogonalization} to remove measured confounding. Implicitly, it is also a \emph{weighting} strategy by using the inverse of the ``rate of change of the treatment" as weights. Three concrete causal DGPs satisfying all the identification assumptions are used as running examples for the exposition of the proposed strategies.
Recently, \citet{ying2022causal} generalized the g-formula and IPTW to the continuous-time general \emph{stochastic treatment regimes} setting.
In contrast, our approach focuses on causal estimands corresponding to \emph{deterministic treatment regimes}. 

\subsection{Setting and notation}

Define a complete probability space $(\Omega, \mathscr{F}, \Prob)$, where $\Omega$ is the sample space, $\fF$ is a $\sigma$-algebra, and $\Prob$ is a probability measure. We consider a finite study period which ends at time $T$, and define the index set $\mathcal{T} = [0, T]$. In the following, each $d$-dimensional Euclidean space $\R^d$ (or its subspace $A \subseteq \R^d$) is equipped with its Borel $\sigma$-algebra $\mathscr{B}(\R^d)$ (\emph{resp.} $\mathscr{B}(A)$) \citep{jacod2013limit,kallenberg1997foundations}. 

A process is \cadlag if all of its paths are right-continuous and admit left-hand limits. For a process $U_t$ admitting a left-hand limit at all $t$, define its left-limit process as $U_{t-} = \lim_{s \rightarrow t, s<t}U_s$.
Define the following \cadlag processes: the treatment process $W = (W^1,\ldots,W^q): \Omega \times \mathcal{T}\rightarrow \R^q$, the outcome process $Y: \Omega \times \mathcal{T}\rightarrow \R$, and the covariate process $Z = (Z^1, \ldots, Z^{p-1}): \Omega \times \mathcal{T}\rightarrow \R^{p-1}$. By convention, we use superscripts from $1$ to $q$ to represent each component of a multivariate process. 
Further denote $X = (X^1, \ldots, X^p) \equiv (Y, Z^1, \ldots, Z^{p-1})$. Denote $\rW(\omega, t) = \{W(\omega, s): 0 \le s \le t\}$ as the treatment trajectory up to time $t$, where $W(\omega, s)$ (or $W_s(\omega)$, whichever is convenient) is the value of the observable treatment trajectory at time $s$ for $\omega \in \Omega$. Similarly we define $\rX(\omega, t) = \{X(\omega, s): 0 \le s \le t\} = (\rY(\omega, t), \rZ(\omega, t))$. In the following, we will omit $\omega$ when it is unambiguous.

The \cadlag potential outcome trajectory for a subject up to time $t$ is defined as 
\[ \rX^{\rw(T)}_t = \{X^{\rw(T)}_s: 0 \le s \le t\} = (\rY^{\rw(T)}_t, \rZ^{\rw(T)}_t), \]
where $\rX^{\rw(T)}_t$ represents the values of the outcome and covariates of a subject, had this subject precisely followed the predetermined static treatment trajectory of interest $\rw(T) \equiv \{w(s): 0 \le 0 \le T\}$. We assume that a trajectory can only depend on the past, i.e., $\rX^{\rw(T)}_t = \rX^{\rw(t)}_t$ for $t<T$. In the following, we also use $\rw$ or $\rw_T$ to denote $\rw(T)$.


The causal estimand of interest is the counterfactual average outcome evaluated at time $T$, $\E[Y^{\rw(T)}(T)]$, under a deterministic treatment plan $\rw(T)$ specified up to $T$. This estimand can be used to compute the average treatment effect (ATE) of a treatment plan $\rw^*(T)$ compared to a baseline treatment plan $\bbzero$ that is constant $0$, i.e. 
\[ \E[Y^{\rw^*(T)}(T) - Y^{\bbzero}(T)], \]  
or a weighted average of counterfactual means, i.e.
\[\int \E[Y^{\rw(T)}(T)] \dx{\Qrob(\rw(T))},\]
where $\Qrob$ is a known distribution over the space of treatment plans $\rw(T)$ under a stochastic intervention. 


\subsection{Informal motivation}
\label{subsec:informal}

We first provide an informal explanation of the main identification results. We will start from the traditional discrete-time setting. Without loss of generality, suppose the values of all relevant processes are observed at $J+1$ equidistant time points, i.e. $\Delta_J \equiv \{t_0 =0, t_1, \ldots, t_J = T\}$ with $t_k = kT/J$. Following the convention in \citet{jacod2013limit}, we use a prime notation to distinguish discrete-time and continuous-time processes.
The observable data is $X^{\prime}_{k} \equiv (Y^{\prime}_{k}, Z^{\prime}_{k}), W^{\prime}_{k}$ at time $t_k$, $k = 0, 1, \ldots, J$, where $W^{\prime}_{k},Y^{\prime}_{k} \in \R, Z^{\prime}_{k} \in \R^{p-1}$. 
For a deterministic static treatment plan $\rw^{\prime}_{J-1} \equiv (w^{\prime}_0, \ldots, w^{\prime}_{J-1})$, the corresponding potential outcome of interest is denoted as $Y_J^{\prime\rw^{\prime}_{J-1}}$. It is the value of the outcome at the end of the study had the individual followed the treatment plan $\rw^{\prime}_{J-1}$ before the $J$th time point. 
Let $\rW^{\prime}_k = (W^{\prime}_0, \ldots, W^{\prime}_k), \rX^{\prime}_k = (X^{\prime}_0, \ldots, X^{\prime}_k)$. 
Define the discrete-time filtration $\FF^{\prime} = (\fF^{\prime}_k) = (\sigma(\rW^{\prime}_k, \rX^{\prime}_k))$. Let $\overbar{\mathbf{0}}'$ be the baseline treatment plan with $w^{\prime}_k \equiv 0$.
The causal estimand is the counterfactual mean $\E[Y_J^{\prime\rw^{\prime}_{J-1}}]$.

We further suppose that this DGP satisfies the following conditions which discrete-time g methods usually assume for causal identification \citep{hernan2020july}.  
\begin{assumbis}{assum:NID}[Discrete-Time Sequential Exchangeability (DTSE)]
	\label{assum:DTSE}
	With observable data $\rW^{\prime}_J, \rX^{\prime}_J$,
	\[\E[W^{\prime}_k|\rY^{\bar{\prime\mathbf{0}'}}_J, \rW^{\prime}_{k-1}, \rX^{\prime}_{k-1}] = \E[W^{\prime}_k|\rW^{\prime}_{k-1}, \rX^{\prime}_{k-1}],\]
	$k = 0, \ldots, J-1$. \footnote{This version is called ``conditional mean independence'', which is less restrictive than the ``conditional independence'' version $W^{\prime}_k \indep \rY^{\prime\bar{\mathbf{0}'}}_J | \rW^{\prime}_{k-1}, \rX^{\prime}_{k-1}$.}
\end{assumbis}

\begin{assumbis}{assum:CTC}[Discrete-Time Sequential Consistency (DTSC)]
	\label{assum:DTSC}
	For a treatment plan $\rw^{\prime}_{J-1}$,
	\[X^{\prime\rw^{\prime}_{J-1}}_{k} = X'_{k}, \text{ if } \rw^{\prime}_{k-1} =  \rW'_{k-1}, k = 1, \ldots, J; X^{\prime\rw^{\prime}_{J-1}}_{0} \equiv X'_{0}.\]
\end{assumbis}


\begin{assumbis}{assum:CSM}[Discrete-Time Causal Structural Model (DTCSM)]
	\label{assum:DTCSM}
	The potential outcome satisfies the following marginal structural model
	\[Y^{\prime\rw^{\prime}_{J-1}}_k = \eta_0 + \eta_1\frac{T}{J}\sum_{i=0}^{k-1}w^{\prime}_i + f(Z^{\prime\rw^{\prime}_{J-1}}_{k-1}) + \epsilon_k , \]
  \[ Z^{\prime\rw^{\prime}_{J-1}}_{k-1} \equiv Z_{k-1},\]
\end{assumbis}
where $f$ is a general function and $\E\epsilon_k = 0$. The structural model in Assumption \ref{assum:DTCSM} is intended to simplify the exposition, and more complex structural models can also be compatible with the approach described in this paper.



Denote the true value of the parameter $\eta_1$ in the structural model as $\eta_1^*$. By Assumptions \ref{assum:DTSC} and \ref{assum:DTCSM},
\[ \E[Y_J^{\prime\rw^{\prime}_{J-1}}] = \E[Y_J^{\prime} - \eta_1^*\frac{T}{J}\sum_{i=0}^{J-1}W^{\prime}_i + \eta_1^*\frac{T}{J}\sum_{i=0}^{J-1}w^{\prime}_i]. \]
Thus, to identify the average potential outcome, it suffices to identify the causal parameter $\eta_1^*$. 
Consider the following function of $\eta_1$:
\begin{equation}
	f(\eta_1) \equiv \E\left[ \sum_{i=1}^J Y'_{i-1}\left(Y_J^{\prime} - \eta_1\frac{T}{J}\sum_{l = 0}^{J-1}W_l'\right) \left\{W'_i - \E(W'_i|\rW'_{i-1}, \rX'_{i-1})\right\} \right].
	\label{eq:eg1_f}
\end{equation}
We claim that solving the equation $f(\eta_1) = 0$ identifies $\eta_1^*$.
This is because
\begin{equation*}
	\begin{aligned}
		f(\eta_1^*) &= \E\left[ \sum_{i=1}^J Y'_{i-1}\left(Y_J^{\prime} - \eta_1^*\frac{T}{J}\sum_{l = 0}^{J-1}W_l'\right) \left\{W'_i - \E(W'_i|\rW'_{i-1}, \rX'_{i-1})\right\} \right]\\
		&\text{(By Assumptions \ref{assum:DTSC} and \ref{assum:DTCSM}, $Y^{\prime\bbzero'}_k =  Y^{\prime}_k - \eta_1^*\frac{T}{J}\sum_{i=0}^{k-1}W^{\prime}_i$)} \\
		&= \E\left[ \sum_{i=1}^J Y'_{i-1} Y^{\prime\bbzero'}_J\left\{W'_i - \E(W'_i|\rW'_{i-1}, \rX'_{i-1})\right\} \right]\\
		&=  \sum_{i=1}^J \E\left[Y'_{i-1} Y^{\prime\bbzero'}_J\left\{\E[W'_i|Y^{\prime\bbzero'}_J,\rW'_{i-1}, \rX'_{i-1} ] - \E(W'_i|\rW'_{i-1}, \rX'_{i-1})\right\} \right]\\
		&(\text{by Assumption \ref{assum:DTSE}})\\
		&= 0.
	\end{aligned}
\end{equation*}

The above identification strategy follows the idea of \emph{orthogonalization}, by partialling out the effect of the observed past history (i.e. $\rW'_{i-1}, \rX'_{i-1}$) on the current treatment assignment $W'_i$. 
Then, by unconfoundedness (Assumption \ref{assum:DTSE}), the remainder $\tilde{W}'_i \equiv W'_i - \E(W'_i|\rW'_{i-1}, \rX'_{i-1})$ will be ``orthogonal'' to any functions $h(Y^{\prime\bbzero'}_J, \rW'_{i-1}, \rX'_{i-1})$. In the case above, $h = Y'_{i-1} Y^{\prime\bbzero'}_J$. 
The orthogonalization idea has appeared in recent causal inference literature. For example, \citet{chernozhukov2018double} give a concrete example of using orthogonalization to estimate treatment effects with cross-sectional data in the presence of high-dimensional confounders. \citet{bates2022causal} propose an orthogonalization form of parametric g-formula for discrete-time longitudinal data by partialling out the effect of the past history on the current covariates $Z'_i$. 

Next, we move towards the continuous-time setting. Define the differences $\delta W'_i \equiv W'_i - W'_{i-1}$ and $\delta A'_i \equiv \E(W'_i|\rW'_{i-1}, \rX'_{i-1}) - W'_{i-1} = \E(\delta W'_i|\rW'_{i-1}, \rX'_{i-1})$. Then, Equation \eqref{eq:eg1_f} can be rewritten as 
\[ f(\eta_1) = \E\left[ \sum_{i=1}^J Y'_{i-1}\left(Y_J^{\prime} - \eta_1\frac{T}{J}\sum_{l = 0}^{J-1}W_l'\right) \left\{\delta W'_i - \delta A'_i\right\} \right].\]
Note that the study period is $[0, T]$ and fixed. Thus, when $J$ tends to infinity, in analogy to how Riemann sums approach Riemann integrals, informally, the above sum will approach an \emph{\Ito stochastic integral} in continuous-time
\[\E\left[ \int_0^T Y_{s-}\left(Y_T - \eta_1\int_0^T W_t \dx{t}\right) \left\{\dx W_s - \dx A_s\right\} \right],\]
which equals zero when evaluated at $\eta_1^*$. As will be seen later, the stochastic process $A$ is the \emph{compensator} of the treatment $W$ with respect to the observable history. Then, the process $M = W-A$ is the remainder of the treatment process after partialling out all the past observable information.
This simple example demonstrates the main ideas behind our identification strategies in Sections \ref{sec:assump} and \ref{subsec:main}. Its relation to weighting is explained in Section \ref{subsec:weight}.







\subsection{More setting and notation}

We introduce more concepts and notations for continuous-time processes that will be needed to develop the proposed identification theory.

It is convenient to use a \emph{filtration} to represent all the available information from the processes up to a given time. For a process $(U_t)$ taking values in a measurable space $(E, \mathcal{E})$, define its natural filtration as $\FF^U = (\fF^U_t)_{t\in\mathcal{T}}$, where ${\fF_t}^{U} = \sigma(U_s: s \le t)$.
A complete filtration is sometimes needed for theoretical developments, but a natural filtration is not necessarily complete. Thus, we define its completion. Let $\mathcal{N}$ be the set of all $\Prob$-zero-measure sets in $\fF$. For a filtration $\mathbb{H} = (\mathscr{H}_t)$, define its completion as the filtration $\tilde{\mathbb{H}} = (\tilde{\mathscr{H}}_t)$ with $\tilde{\mathscr{H}}_t = \mathscr{H}_t \bigvee \mathcal{N}$, which is the smallest complete filtration such that $\mathscr{H}_t \subseteq \tilde{\mathscr{H}}_t$. Here, $A \bigvee B$ is the smallest sigma-field generated by $A$ and $B$, i.e. $\sigma(A\bigcup B)$, where $A,B\subseteq \fF$.

Define $\FF = (\fF_t)$ with $\fF_t \equiv \fF_t^{W,X}$. Thus, $\fF_t$ represents all the information available from the observable data $W, Y, Z$ up to time $t$. Let $\sigma(\rY^{\rw(T)}_T)$ be the $\sigma$-field generated by the potential outcome trajectory. It contains all the information of the potential outcome under the treatment plan $\rw(T)$ within the study period $\mathcal{T}$.
Define the initially enlarged filtration $\FG = (\fG_t)$, where 
\[\fG_t = \fF_t \bigvee \sigma(\rY^{\bar{\mathbf{0}}}_T),\] 
and $\rY^{\bar{\mathbf{0}}}_T$ is the potential outcome trajectory when the treatment of interest is the constant function $0$ on $\mathcal{T}$, i.e. $w_s = 0$ for all $s$. 
Intuitively, if one has access to $\fG_t$, then one not only knows the trajectories of $W,Y,Z$ and the potential outcome $Y^{\bar{\mathbf{0}}}$ up to time $t$, but can also peek into future values of the potential outcome up to the end of the study. Therefore, we call $\FF$ the \emph{factual} (or \emph{observable}) \emph{filtration}, and $\FG$ the \emph{counterfactual filtration}.


Data in statistics can often be described as the sum of signal and noise. This idea is captured by the notion of a special semimartingale \citep{kallenberg1997foundations} for stochastic processes.
We focus on the case where the treatment process $W$ is an $\mathbb{F}$-special semimartingale. 
Formally, given a generic filtration $\FH = (\fH_t)$, a real-valued process $(U_t)$ is called an $\FH$-semimartingale if \[U_t = U_0 + M^U_t + A^U_t,\] where $M^U_0 = A^U_0 = 0$, $U_0$ is $\fH_0$-measurable, $M^U_t$ is an $\FH$-local martingale, and $A^U_t$ is a \cadlag $\FH$-adapted finite variation process. Furthermore, a semimartingale $U_t$ is an $\FH$-special semimartingale if $A^U_t$ is also $\FH$-predictable. Then the decomposition is unique (i.e. the \emph{canonical decomposition}), and $A^U$ is called the \emph{compensator} of $U$. 

Intuitively, $A^U_t$ reflects the infinitesimal systematic change of the treatment process given the past history, while $M^U_t$ reflects the remaining variation (i.e., the ``random noise"). An $\R^k$-valued process is a (special-) semimartingale if each of its components is a (special-) semimartingale. Examples of special semimartingales include submartingales or supermartingales, counting processes, and continuous semimartingales (e.g. diffusions and \Ito processes). We focus on special semimartingales for two reasons: (1) we can resort to the powerful results from \Ito calculus, since the semimartingales form the largest class of processes as integrators for which the \Ito integral can be defined; and (2) the uniqueness of the decomposition of special semimartingales avoids unnecessary technicalities for practical applications.

For convenience, we summarize our setting:
\begin{assum}[Setting]
	We assume \\
	(1) Processes $W_t, X_t$, the static treatment plan $w_t$, and the potential outcome process $X^{\rw}_t$ are \cadlag. \\ 
	(2) $W$ is an $\mathbb{F}$-special semimartingale. 
	\label{assum:setting}
\end{assum}




\section{Examples}

We present three classes of continuous-time DGPs as running examples. We will verify our identification assumptions for each of them later, showing that the statistical model $\mathcal{M}$ is nontrivial. Figure \ref{fig:realization} shows realizations from each example DGP.

\begin{figure}
	\centering
	\includegraphics[width = \linewidth]{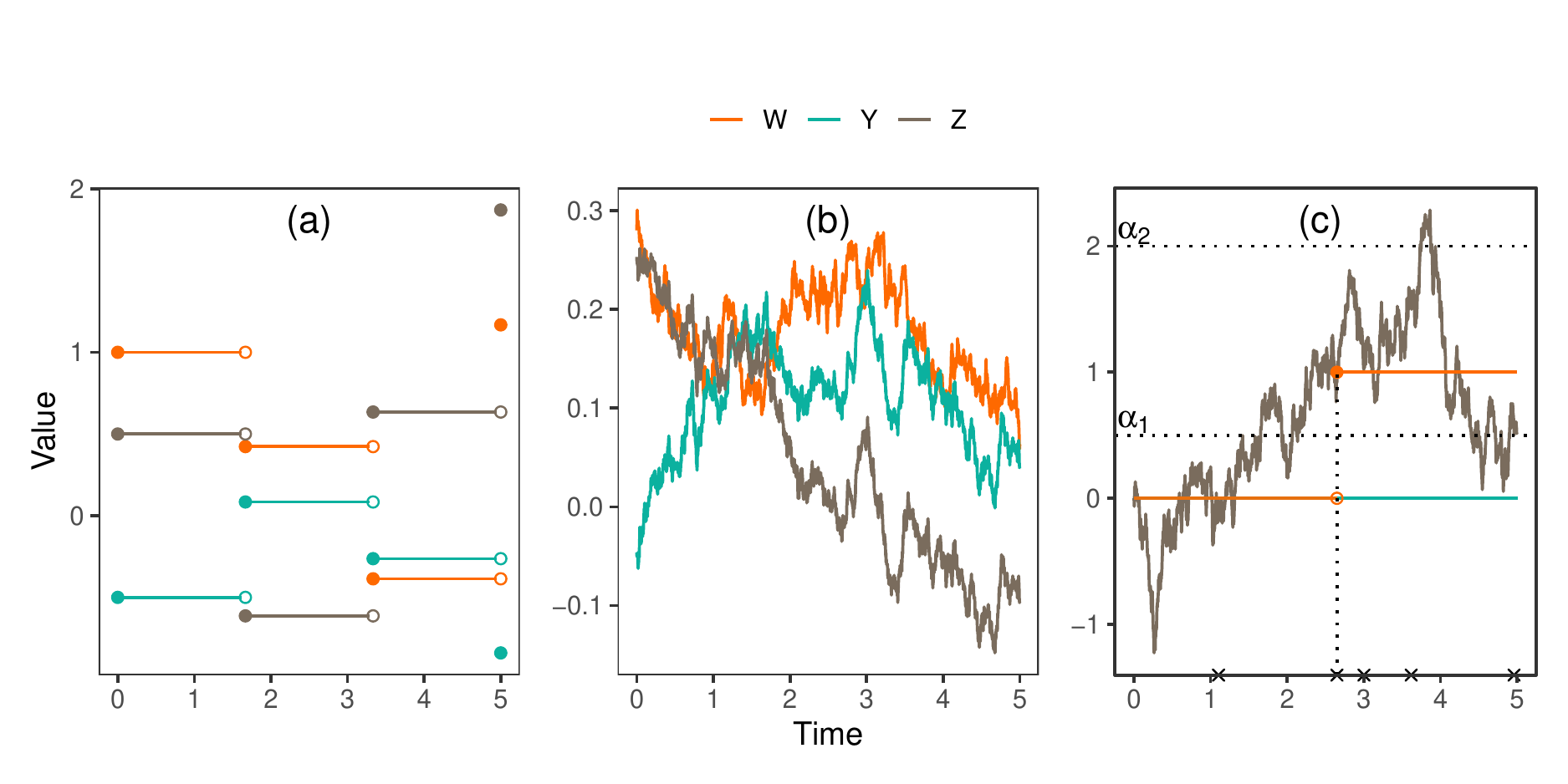}
	\caption{ Realizations of Example 1 (Panel a), Example 2 (Panel b), and Example 3 (Panel c). Crossings on the x-axis of Panel c are the jump times of the driving Poisson process $N$. ($W$: treatment, $Y$: outcome, $Z$: covariate.) }
	\label{fig:realization}
\end{figure}

\subsection{Example 1: Discrete-time DGPs}
\label{subsec:eg1}

This is the example described in Section \ref{subsec:informal}, where the true DGP is discrete-time. Any discrete-time DGP has a unique corresponding continuous-time representation \citep[Chap.~1]{jacod2013limit} as step functions, and thus the discrete-time DGP is a special case of our general setting. Recall that the causal estimand is $\E[Y_J^{\prime\rw^{\prime}_{J-1}}]$, the discrete-time filtration is $\FF^{\prime} = (\fF^{\prime}_k) = (\sigma(\rW^{\prime}_k, \rX^{\prime}_k))$, and $\FG^{\prime} = (\fG^{\prime}_k)$ with $\fG^{\prime}_k \equiv \fF^{\prime}_k \bigvee \sigma(\rY_J^{\prime\overbar{\mathbf{0}}'})$. We will focus on its equivalent continuous-time representation below.
A realization of $W,X$ under this DGP is shown in Figure \ref{fig:realization}(a).

Define the continuous-time processes $W_t \equiv W^{\prime}_{\floor*{Jt/T}}, X_t \equiv X^{\prime}_{\floor*{Jt/T}}, w_t = w^{\prime}_{\floor*{Jt/T}}, Y_t^{\rw} = Y_{\floor*{Jt/T}}^{\prime\rw^{\prime}_{J-1}}$, which are step functions of $t \in [0, T]$. Here $\floor*{r}$ is the largest integer that is not greater than the real number $r$.
Define ${\mathbb{F}} = ({\fF_t})_{t \in \mathcal{T}}$ with $\fF_t = \fF_t^{W,X}$. $\mathbb{F}$ is the history of all the observed information up to each time $t$. Define $\mathbb{G} = (\fG_t)$ with $\fG_t = \fF_t \bigvee \sigma(\rY^{\bar{\boldsymbol{0}}}_T)$. The $\sigma$-algebra $\fF_t, \fG_t$ will remain the same between two observation time points since no new information is observed in between. By \citet[Chap.~1.4]{jacod2013limit}, since $W'$ is $\FF'$-adapted, $W$ is an $\FF$-semimartingale. Then, given $\E\left[\sum_{1\le i \le J} |W'_i - W'_{i-1}|\right] < \infty$, $W$ will be an $\FF$-special semimartingale. Thus, Assumption \ref{assum:setting} is satisfied.


\subsection{Example 2: Continuously-valued treatment and outcome}
\label{subsec:eg2}

Many variables in observational studies are continuous, e.g. blood pressures and household incomes. 
We show an example where $W_t, X_t$ are continuous random variables. We use the multivariate Ornstein-Uhlenbeck process \citep[Chap.~1]{iacus2009simulation}, a type of diffusion processes, as the DGP. It has wide applications in physics \citep{uhlenbeck1930theory}, evolutionary biology \citep{hunt2007relative}, and mathematical finance \citep{leung2015optimal, bjork2009arbitrage}.

Let $\boldsymbol{B}_t = (B_{t}^1, B_{t}^2)$ be a $2$-dimensional standard Brownian motion with respect to its natural filtration. Define $Y_t, W_t, Z_t$ as the outcome process, the treatment process, and the covariate process, with initial random vector $(Y_0,  W_0, Z_0)^{\intercal}$ which is independent of $B^1, B^2$. Define $\mathbb{H} = (\mathscr{H}_t)_{t \in \mathcal{T}}$ where $\mathscr{H}_t \equiv \fF_t^{B^1, B^2} \bigvee \mathcal{N} \bigvee \sigma(Y_0,  W_0, Z_0)$. We consider a 3-dimensional OU process \citep{gardiner1985handbook} as the causal mechanism, which is the unique $\mathbb{H}$-strong solution of the following SDE:
\begin{equation}
	\dx{\begin{pmatrix*} Y_t \\ W_t \\Z_t \end{pmatrix*}} = -\boldsymbol{\beta}\begin{pmatrix*} Y_t \\ W_t \\ Z_t \end{pmatrix*}\dx{t} + \boldsymbol{\sigma}\dx{\boldsymbol{B}_t},
	\label{eq:3dOU}
\end{equation}
where 
\[ \boldsymbol{\beta} =\begin{pmatrix}
	&\beta_{11}, &\beta_{12}, &\beta_{13} \\
	&\beta_{21}, &\beta_{22}, &\beta_{23} \\
	&\beta_{31}, &\beta_{32}, &\beta_{33} 
\end{pmatrix}  \text{ and }
\boldsymbol{\sigma}= \begin{pmatrix}
	&\sigma_{11}, &0\\
	&0, &\sigma_{22} \\
	&\sigma_{31}, &0
\end{pmatrix} \]
are constant matrices.  Intuitively, the \emph{drift} parameter $\boldsymbol{\beta}$ describes the instantaneous expected influence between observable processes, while the \emph{diffusion} parameter $\boldsymbol{\sigma}$ specifies the structure and magnitudes of the impact of the external noises $\boldsymbol{B}$ on observable processes. For example, the parameter $\boldsymbol{\sigma}$ in Equation \eqref{eq:3dOU} indicates that the exposure and outcome processes do not have common external sources of influence. This specification has important implications with regard to the unconfoundedness condition, as will be shown later.
A realization of $W,X$ under this DGP is shown in Figure \ref{fig:realization}(b).


Given an arbitrary treatment trajectory of interest $\rw(T)$, the potential outcome process is the strong solution of
\begin{equation}
	\begin{aligned}
		\dx{Y_t^{\rw(T)}} &= -\left(\beta_{11} Y_t^{\rw(T)} + \beta_{12} w_t + \beta_{13}Z_t^{\rw(T)} \right)\dx{t} + \sigma_{11}\dx{B}^1_{t} \\
		\dx{Z_t^{\rw(T)}} &= -\left(\beta_{31} Y_t^{\rw(T)} + \beta_{32} w_t + \beta_{33}Z_t^{\rw(T)} \right)\dx{t} + \sigma_{31}\dx{B}^1_{t}\\
	\end{aligned}
	\label{eq:ou_po}
\end{equation}
with $Y_0^{\rw(T)} \equiv Y_0, Z_0^{\rw(T)} \equiv Z_0$.
Suppose $Y, W, Z$ are all observable, while $B_{t}^1, B_{t}^2$ are unmeasured, then we have observable filtration $\mathbb{F} = (\fF_t)$ with $\fF_t = \fF_t^{Y, W, Z}$. The enlarged filtration is $\mathbb{G} = (\fG_t), \fG_t = \fF_t \bigvee \sigma(\rY_T^{\bar{\boldsymbol{0}}})$. Since $W$ is a diffusion process, it is an $\FF$-special semimartingale, and therefore Assumption \ref{assum:setting} is satisfied.


\subsection{Example 3: Time-to-event treatment and outcome}
\label{subsec:eg3}

Time-to-event variables are common in public health and social sciences, e.g. disease occurrences and financial defaults. In the following, we describe a causal DGP where both the treatment and outcome are time-to-event data and the covariate is a continuous process. For example, the treatment is the time-to-initiation of a medical treatment, and the outcome is the time to an adverse health event.

Let $N$ and $B$ be a standard Poisson process and a standard Brownian motion respectively, with respect to their natural filtrations. Define $Y, W, Z$ as the outcome, treatment and covariate processes with $Y_0 = W_0 = Z_0 = 0$, which are all observable. Define the random time at which $Z_t$ first rises above a threshold $\alpha_1$ as $\iota^1 = \inf\{t\in\mathcal{T}: Z_t \ge \alpha_1\}$, and the time at which $Z_t$ rises above another threshold $\alpha_2$ as $\iota^2 = \inf\{t\in\mathcal{T}: Z_t \ge \alpha_2\}$, and by convention let $\inf\emptyset = \infty$. Define the corresponding indicator processes $U^1_t = \indicator{t \ge \iota^1}$ and $U^2_t = \indicator{t \ge \iota^2}$.
Define $\FH = (\fH_t)$ with $\fH_t = \fF^{N,B}\bigvee \mathcal{N}$. 
The causal DGP is the $\FH$-strong solution of the following SDE:
\begin{equation*}
\begin{split}
	\dx{Z_t} &= \alpha_0\dx{B_t} \\
	\dx{W_t} &= U^1_{t-}(1 - W_{t-})\dx{N_t} \\ 
	\dx{Y_t} &= (1 - W_{t-})\dx{U^2_t}
	\end{split}
\end{equation*}

Here, $\boldsymbol{\alpha} = (\alpha_0, \alpha_1, \alpha_2)$ are model parameters, where $0 < \alpha_1 < \alpha_2$. As an interpretation of the above DGP, $Z$ represents the severity of a disease whose volatility of progression is characterized by $\alpha_0$, and when $Z$ passes a certain threshold (i.e. $\alpha_1$) at time $\iota^1$, the subject will start to seek medical treatment and wait for a random time before the procedure/surgery or medication become available.
If the disease progresses to $\alpha_2$ at time $\iota^2$ before the medical treatment starts, then the subject will experience a severe adverse event. Since $W_t$ is a counting process, Assumption \ref{assum:setting} is obviously satisfied. A realization of $W,X$ is shown in Figure \ref{fig:realization}(c).

For a one-jump treatment plan $\rw_T$, the potential outcome process is
\begin{equation*}
\begin{split}
		\dx{Z^{\rw}_t} &= \alpha_0\dx{B_t} \\
	\dx{Y}^{\rw}_t &= (1 - w_t)\dx{U^{2, \rw}_t},
	\end{split}
\end{equation*}
where $U^{2, \rw}_t = \indicator{t \ge \iota^{2,\rw}}$ and $\iota^{2, \rw} = \inf\{t\in\mathcal{T}: Z^{\rw}_t \ge \alpha_2\}$. We thus define the observable filtration $\FF = (\fF_t)$ with $\fF_t = \fF_t^{Y, W, Z}$, and the expanded filtration $\FG = (\fG_t)$ with $\fG_t = \fF_t \bigvee \rY^{\bar{\boldsymbol{0}}}_T$.


\section{Definition of identification}

Before describing our identification strategies, we need to first define the meaning of ``identification" in this paper. Following \citet{lewbel2019identification, basse2020general}, we now formalize what \emph{identification} entails generally.
Define the statistical model as $\mathcal{M}$, which contains all the DGPs that satisfy the specified assumptions. 
For a DGP $m \in \mathcal{M}$, a function $\Psi(m)$ defines the estimand of interest. Define $\Theta = \{\Psi(m): m \in \mathcal{M}\}$ as the set of possible estimand values. Define function $\Phi(m)$ mapping the DGP to observable information, and $\Xi = \{\Phi(m): m \in \mathcal{M}\}$ as the image of $\mathcal{M}$. Denote the true DGP as $m_0$, and $\theta_0 = \Psi(m_0), \xi_0 = \Phi(m_0)$. Thus $\xi_0$ is what we can observe from the data.
\begin{defn}
	An \emph{identification strategy} is a function $g: \Xi\times \Theta \rightarrow \R^k$ such that 
\[g\left(\Phi(m), \Psi(m)\right) = \boldsymbol{0}, \forall m \in \mathcal{M}.\] 
	\label{defn:id}
\end{defn}
In other words, $g$ is an identification strategy if and only if $g$'s set of zeros $\Lambda(g) = \{(\xi, \theta)\in \Xi\times \Theta: g(\xi, \theta) = 0\}$ is a superset of the image of $\mathcal{M}$ under the function $(\Phi, \Psi)$, $\Lambda_0 \equiv \{(\Phi(m), \Psi(m)): m \in \mathcal{M}\}$, i.e. \[\Lambda_0 \subseteq \Lambda(g).\] 
We call an identification strategy $g$ \emph{sharp} if $\Lambda(g) = \Lambda_0$. We say that $g$ point-identifies $\theta_0$ when $g\left(\xi_0, \Psi(m)\right) = 0$ if and only if $\Psi(m) = \Psi(m_0), \forall m_0 \in \mathcal{M}$; otherwise, we say $g$ partially identifies $\theta_0$. 

For example, for a metric space $(\Theta, d)$, when there is an explicit function $\tilde{g}$ such that $\Psi(m) = \tilde{g}(\Phi(m))$, the corresponding identification strategy $g(\xi, \theta) \equiv d(\tilde{g}(\xi), \theta)$ is sharp and point-identifies $\theta_0$. 
Figure \ref{fig:id} is a visual representation of the above concepts. 
In panel (b), the black area in $\Xi\times\Theta$ is $\Lambda_0$, and the grey area is $\Lambda(g)$. Since $\Lambda(g)$ is a strict superset of $\Lambda_0$, $g$ is not sharp; since a single value of $\xi \in \Xi$ may correspond to many different $\theta$'s such that $(\xi, \theta) \in \Lambda(g)$, $g$ partially identifies $\theta_0$.

The causal identification is a special case of the above definition where the estimand is a causal estimand that is related to counterfactual variables. In particular, in our causal question, different DGP $m$'s correspond to different functions $W,X,\{Y^{\rw}\}_{\rw\in \mathcal{W}}$ mapping from $\Omega\times \mathcal{T}$, where $\mathcal{W}$ is the set of all treatment plans. We omit the notation showing this dependence on $m$ when it is not ambiguous. Then, the causal estimand is $\Psi(m) \equiv \E[Y_T^{\rw}]$ and the observable information is $\Phi(m) = \{(W, X)\} = \{(W(\omega, t), X(\omega, t)): \omega \in \Omega, t \in \mathcal{T}\}$. Knowing the function $(W,X)$ implies the usual saying that ``an infinite number of i.i.d. realizations of $(W_t, X_t)$ are available''. 
In Section \ref{sec:assump}, we will introduce the assumptions that specify the statistical model $\mathcal{M}$, and in Section \ref{sec:id}, we will give a class of identification strategies $g$.

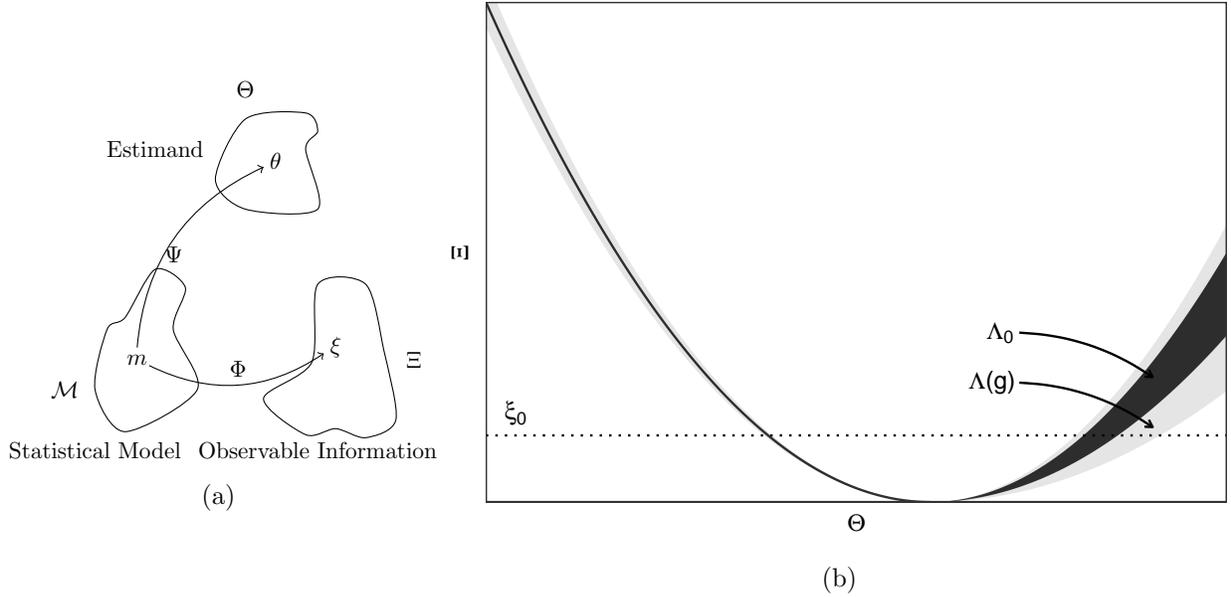
\begin{figure}
	\centering
	\begin{subfigure}{.35\textwidth}
		\centering
		\begin{tikzpicture}[scale=0.8, transform shape]
			\draw plot [smooth cycle] coordinates {(5,0.25) (5.5,0.35) (6, 0.2) (6.5,0.5) (6.3,1.65) (6,2.75) (5.2,2.75) (5.1,1.45) (4.3,0.85) } node at (5.5,1.7) {$\xi$};
			\draw plot [smooth cycle] coordinates {(1.5,1)(1.7,2)(2,2.2)(2.5,3)(3,2.7)(2.8,2)(3.2,1)(2,0.3)} node at (2.2,1.5) {$m$};
			\draw plot [smooth cycle] coordinates {(3.5,4.5) (4,5.5) (5, 5.6) (5.2,5.3) (5,5) (5.2,4) (4,4) } node at (4.5,4.8) {$\theta$};
			\node at (1,1) {$\mathcal{M}$};
			\node at (1.5,0) {Statistical Model};
			\node at (6.8,1.5) {$\Xi$};
			\node at (5.2,0) {Observable Information};
			\node at (4,6) {$\Theta$};
			\node at (2.5,5) {Estimand};
			\path[->] (2.2,1.7) edge [bend left] node[below] {$\Psi$} (4.3,4.7);
			\path[->] (2.4,1.4) edge [bend right] node[above] {$\Phi$} (5.3,1.6);
		\end{tikzpicture}
		\caption{}
		\label{fig:id1} 
		\end{subfigure}%
		\begin{subfigure}{.65\textwidth}
		\centering
		\includegraphics[width = 0.99\textwidth]{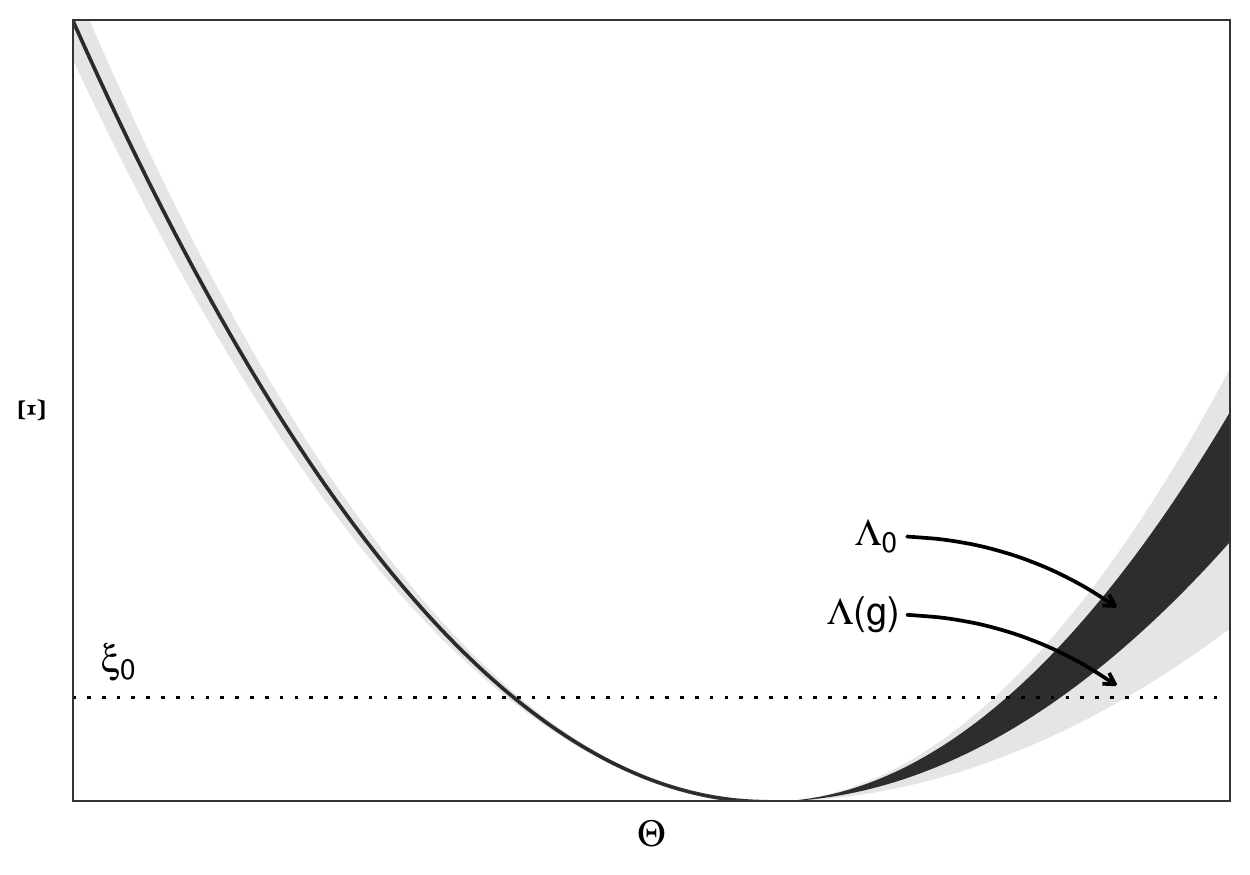}
		\caption{}

		\label{fig:id2}
		\end{subfigure}
	\caption{(a) Illustration of mappings between the statistical model $\mathcal{M}$, the estimand space $\Theta$, and the observable information set $\Xi$. $\Theta$ and $\Xi$ are images of $\mathcal{M}$ under functions $\Psi$ and $\Phi$, respectively. (b) Illustration of an identification strategy $g$ which is not sharp and partially identifies the estimand. Black area: $\Lambda_0$; Grey area: $\Lambda(g)$.} 
	\label{fig:id}
\end{figure}


\section{Identification assumptions}
\label{sec:assump}

In this section,  we introduce causal identification assumptions that place restrictions on the statistical model $\mathcal{M}$ for the causal DGP of continuous-time processes. Broadly, the assumptions generalize standard assumptions from the discrete-time g-methods literature: sequential exchangeability, sequential consistency, sequential positivity, and structural models \citep{hernan2020july}.

\subsection{No Information Drift (NID)}


In observational studies, when there is unmeasured confounding, the assignment of treatment at time $t$ depends not only on the observable information before time $t$, i.e. $\fF_{t-}$, but may also depend on  information about the potential outcome in the future (e.g. $\sigma(\rY^{\bar{\mathbf{0}}}_T)$). For causal identification, we need a ``no unmeasured confounding''-type condition. The following condition is a continuous-time generalization of the ``sequential exchangeability'' condition in the discrete-time g methods literature \citep[Chap.~19]{hernan2020july}. Recall that the $\FF$-canonical decomposition of $W$ is $W_t = W_0 + M^W_t + A^W_t$, where $M^W$ is the local martingale part, and $A^W$ is the compensator.

\begin{assum}[No Information Drift (NID)]
	\label{assum:NID}
	The canonical decompositions of the treatment process $W$ with respect to $\mathbb{G}$ and $\mathbb{F}$ are the same.
\end{assum}

Assumption \ref{assum:NID} indicates an \emph{invariance} under a \emph{change of filtration}. It is equivalent to the statement that $M^W$ is both an $\mathbb{F}$- and a $\mathbb{G}$-local martingale, since if $A^W$ is an $\mathbb{F}$-predictable finite variation process, it is also a predictable finite variation process with respect to $\mathbb{G}$. 

\emph{Information drift} is a well-known concept in the stochastic analysis literature \citep{ankirchner2006shannon}.
Generally, under an enlargement of filtration, $M^W$ may not be a $\mathbb{G}$-local martingale, or even a $\mathbb{G}$-semimartingale. When $W$ is a real-valued continuous semimartingale and $M^W$ is a $\mathbb{G}$-semimartingale, \citet{ankirchner2006shannon} shows that 
\[M^W = \tilde{M}^W + \int\alpha_s\dx{[M^W]_s},\]
where $\tilde{M}^W$ is a $\mathbb{G}$-local martingale, $[M^W]_s$ is the quadratic variation process of $M^W$, and $\alpha_s$ is a unique (up to indistinguishability) $\mathbb{G}$-predictable process. The integrand $\alpha_s$ is often called the \emph{information drift} of the expansion from $\mathbb{F}$ to $\mathbb{G}$. It reflects how the additional information contained in $\mathbb{G}$ but not in $\mathbb{F}$ changes the systematic behavior of $M^W$. In our setting, $\alpha_s$ reflects how \emph{the asymmetric information} related to the potential outcome $\sigma(\rY^{\bar{\mathbf{0}}}_T)$ that treatment decision-makers use (consciously or unconsciously) but the observational data does not include, will influence the assignment mechanism of treatment. This translates the concept of unmeasured confounding to the concept of asymmetric information.  Assumption \ref{assum:NID} implies that 
\[\alpha_s \equiv 0, \forall s,\] 
which means no information drift and implies no unmeasured confounding. 


When $T_W$ is a time-to-event treatment, e.g. treatment initiation or termination, and the treatment process is defined as $W = \indicator{t \ge T_W}$, the No Unmeasured Confounding conditions in existing continuous-time causal inference literature (e.g. \citet{zhang2011causal,hu2019causal,yang2021semiparametric2}) are special cases of Assumption \ref{assum:NID}, using decompositions of random times.

Based on Assumption \ref{assum:NID}, we formally propose a definition of ``confounders'' for trajectories below, in the sense that these covariates form a ``sufficient adjustment set'' rendering NID satisfied.
\begin{defn}
	For the treatment process $W$, and the outcome process $Y$, a set of covariate processes $Z$ is called \emph{confounders}, if the canonical decompositions of $W$ under the filtration $\left(\fF^{W,Y,Z}_t\right)$ and the filtration $\left( \fF^{W,Y,Z}_t \bigvee \sigma(\rY_T^{\bar{\mathbf{0}}}) \right)$ are the same.
\end{defn}

Next, we verify Assumption \ref{assum:NID} for Example 1-3. These concrete examples show the connections to the discrete-time setting, and provide insights into continuous-time confounding.
We will show that discrete-time sequential exchangeability (Assumption \ref{assum:DTSE}) is a special case of no information drift (Assumption \ref{assum:NID}). However, no information drift does not imply discrete-time sequential exchangeability unless $M^W$ is both an $\FF$- and $\FG$-martingale instead of a local martingale. Similarly, we verify that there is no information drift in DGPs in Examples 2 and 3. We formally state the results in Proposition \ref{prop:eg1}, \ref{prop:eg2}, and \ref{prop:eg3}. Proofs are presented in Appendix A. 

\begin{prop}[NID for Example 1]
	Using notations in Section \textup{\ref{subsec:eg1}},\\
	(i) Assumption \textup{\ref{assum:DTSE}} implies Assumption \textup{\ref{assum:NID}}. \\
	(ii) If $M^W$ is both an $\FF$- and $\FG$-martingale, then Assumption \textup{\ref{assum:DTSE}} holds.
	\label{prop:eg1}
\end{prop}

\begin{prop}[NID for Example 2]
	Using notations in Section \textup{\ref{subsec:eg2}}, the canonical decompositions of $W_t$ with respect to $\mathbb{F}$ and $\mathbb{G}$ are the same. 
	\label{prop:eg2}
\end{prop}
\begin{prop}[NID for Example 3]
	Using notations in Section \textup{\ref{subsec:eg3}}, the canonical decompositions of $W_t$ with respect to $\mathbb{F}$ and $\mathbb{G}$ are the same.
	\label{prop:eg3}
\end{prop}

	


\subsection{Continuous-Time Consistency Assumption (CTC)}

Generalizing discrete-time sequential consistency assumption\citep[Chap.~19]{hernan2020july}, we make the following assumption for the continuous-time case to link the potential outcomes to observable outcomes.

\begin{assum}[Continuous-Time Consistency (CTC)]
	\label{assum:CTC}
	For all treatment trajectories $\rw(T)$, if $\rW_{t} = \rw_{t}$ for $t \in (0, T]$, then we observe 
	$\rY_t = \rY^{\rw(T)}_{t}$  and $ \bar{Z}_t = \bar{Z}^{\rw(T)}_{t}$. Furthermore, $Y^{\rw(T)}_{0} = Y_{0}$ and $Z^{\rw(T)}_{0} = Z_{0}$.
\end{assum}

Assumption \ref{assum:CTC} requires that there is no interference between individuals, and there are no multiple versions of the treatment. Examples 1-3 clearly satisfy this assumption.


\subsection{Causal Structural Model (CSM)}

For a treatment plan $\rw_T$, define $\tau_t(\rw_T) \equiv Y^{\rw_T}_t - Y^{\bar{\mathbf{0}}}_t$ as the individual treatment effect of $\rw_T$ at time $t$. By Assumption \ref{assum:CTC}, we have $\tau_0(\rw_T) \equiv 0$. We make the following causal structural assumption on $\tau_t(\rw_T)$, which depicts the relationship between a specified intervention and the potential outcome. Note that the baseline potential outcome process $Y^{\bar{\mathbf{0}}}_t$ is left unrestricted.

\begin{assum}[Causal Structural Model (CSM)]
	The treatment effect $\tau_t(\rw_T)$ is $\FF$-adapted and is known up to the causal parameter $\gamma \in \Gamma$, i.e. $\tau_t(\rw_T) = \tau_t(\rw_T;\gamma)$.
	\label{assum:CSM}
\end{assum}

For example, the treatment effect might be represented by a function of the cumulative treatment delivered, e.g.  $\tau_t(\rw_T;\gamma) = \int_0^t \gamma(s,t)w_s \dx{s}$, where the unknown function $\gamma$ describes how the the effects of prior treatments accrue. Another example for a counting process outcome is $\tau_t(\rw_T;\gamma) = \int_0^t \gamma(s,t)w_s \dx{N_s}$, where the counting process $N_s$ is adapted to $\FF$ and is known from the observable data. Note that $\tau_t(\rw_T;\gamma)$ can be different among subjects, and $\Gamma$ can have finite dimensions (e.g. Euclidean spaces) or infinite dimensions (e.g. function spaces).

Suppose that the filtration $\FH = (\fH_t)$ records all the relevant information for the causal system and thus $\fF_t \subseteq \fH_t$. Then in general, $\tau_t(\rw_T)$ is $\FH$-adapted, but not necessarily $\FF$-adapted. However, in many cases, Assumption \ref{assum:CSM} may be satisfied. For example, the unobserved information may cancel out when we subtract $Y^{\bar{\mathbf{0}}}_t$ from $Y^{\rw_T}_t$. 
We verify Assumption \ref{assum:CSM} for Examples 1-3.

\paragraph{Example 1 (Continued)}
By Assumption \ref{assum:DTCSM}, $Y^{\prime\rw^{\prime}_{J-1}}_k = \eta_0 + \eta_1\frac{T}{J}\sum_{i=0}^{k-1}w^{\prime}_i + f(Z'_{k-1})  + \epsilon_k$, and $Y^{\prime\rw^{\prime}_{J-1}}_k - Y^{\prime\bbzero'}_k =  \eta_1\frac{T}{J}\sum_{i=0}^{k-1}w^{\prime}_i$. Then the corresponding continuous-time representation of the treatment effect becomes
\begin{equation}
	\tau_t(\rw) = \eta_1\int_0^{\floor*{\frac{Jt}{T}}\frac{T}{J}} w_s\dx{s},
	\label{eq:eg1_csm}
\end{equation}
which is deterministic and thus adapted to $\FF$, and parameterized by a 1-dimensional $\gamma = \eta_1$.

\paragraph{Example 2 (Continued)}
Define the causal estimand as $\tau_t(\rw_T)\equiv \tau_t^Y(\rw_T)   \equiv Y^{\rw_T}_t - Y^{\bar{\mathbf{0}}}_t$, and for convenience define $\tau_t^Z(\rw_T) \equiv Z^{\rw_T}_t - Z^{\bar{\mathbf{0}}}_t$. 
Then by \citet{gardiner1985handbook}, the solution to Equation \eqref{eq:ou_po} is 
\begin{equation}
	\begin{pmatrix*} \tau_t^Y(\rw_T)\\ \tau_t^Z(\rw_T) \end{pmatrix*} = \int_0^t e^{\left(\begin{smallmatrix*} \beta_{11} & \beta_{13}\\ \beta_{31} & \beta_{33} \end{smallmatrix*}\right)(s-t) } \begin{pmatrix*} -\beta_{12} w_s \\ -\beta_{32} w_s\end{pmatrix*} \dx{s},
	\label{eq:eg2_csm}
\end{equation}
where the first term of the integrand is a matrix exponential. Equation \eqref{eq:eg2_csm} is deterministic, and thus $\tau_t^Y(\rw_T)$ is $\FF$-adapted. The $\tau_t^Y$ and $\tau_t^Z$ are parameterized by $\gamma=(\beta_{11}, \beta_{12}, \beta_{13}, \beta_{31}, \beta_{32}, \beta_{33})$.

\paragraph{Example 3 (Continued)}
Since $Z$ is not influenced by $W$, we have $U^{2, \rw}_s \equiv U^2_s$. Then it is easy to see that 
\begin{equation}
	\tau_t(\rw_T) = Y_t^{\rw_T} - Y_t^{\bbzero} = -\int_0^t w_{s-}\dx{U^2_s}.
	\label{eq:eg3_csm}
\end{equation} 
Since $U^2$ is knowable from observable data and is $\FF$-adapted, then $\tau_t(\rw_T)$ is $\FF$-adapted, and is parameterized by $\gamma = (\alpha_0, \alpha_2)$.


\subsection{Continuous-Time Positivity (CTP)}

In discrete time, the sequential positivity condition\citep[Chap.~19]{hernan2020july} is needed so that sufficient randomness exists in the treatment assignment at each time point after conditioning on the past.  As mentioned above, $M^W$ is the ``noisy'' component of the treatment. Thus, we extend the sequential positivity condition to the continuous-time setting by imposing the following condition on $M^W$. 

\begin{assum}[Continuous-Time Positivity A (CTP-A)] 
	The local martingale $M^W$ is not constant zero.
	\label{assum:CTP-A}
\end{assum}

Assumption \ref{assum:CTP-A} is equivalent to requiring that $W$ or $M^W$ is not an $\FF$-predictable finite variation process. This assumption will be violated, for example, if $W_t$ has continuously differentiable sample paths. In real-life applications, inevitable noises in the treatment assignment often render the sample paths ``rough'' and thus ensure the satisfaction of Assumption \ref{assum:CTP-A}. Examples 1-3 clearly satisfy this assumption because the treatments are not $\FF$-predictable finite variation processes.


\section{Identification results}
\label{sec:id}

We will first present our main identification strategies through orthogonalization, and then explain the connection to the weighting approach.

\subsection{Main identification strategy}
\label{subsec:main}

Suppose the statistical model $\mathcal{M}$ satisfies Assumptions \ref{assum:setting} to \ref{assum:CTP-A}, and the true causal DGP $m_0$ is in $\mathcal{M}$. Recall that the probability space is fixed, and that the measurable functions $W,X,X^{\rw}$ depend on the DGP $m_0$, where we omit the explicit notation of this dependence for conciseness. Our observable information is $\xi_0 = \Phi(m_0) = \{(W, X)\} = \{(W(\omega, t), X(\omega, t)): \omega \in \Omega, t \in \mathcal{T}\}$, and our estimand of interest is $\theta_0 = \Psi(m_0) \equiv \E[Y_T^{\rw}]$. 
We will develop a moment-based identification strategy $g$ such that $g\left(\Xi(m_0), \Psi(m_0)\right) = 0, \forall m_0 \in \mathcal{M}$. In fact, to identify $\E[Y_T^{\rw}]$, we only need to identify $\gamma_0$, the true causal parameters in $\tau_t(\rw_T;\gamma_0)$ under $m_0$. 

Recall that $M^W \equiv (M^{W^1}, \ldots, M^{W^q})$ is the $q$-dimensional local martingale part of the $\mathbb{F}$-canonical decomposition of $W_t$, and define $[X, Y]_s$ to be the quadratic covariation process of two generic processes $X, Y$. Define $[l] \equiv \{1,2, \ldots, l\}$ for an integer $l$. Define the $d$-dimensional process $H \equiv (H^1, \ldots, H^d): \Omega \times \mathcal{T} \times \Gamma \rightarrow \R^d$ with 
\begin{equation}
    H^i_t(\gamma) \equiv h^i\left(Y_T - \tau_T(\rW_T;\gamma), W_{t-}, X_{t-}, t\right)
    \label{eq:H}
\end{equation}
and $h_i$'s are Borel-measurable functions from $\R\times \R^q \times \R^p \times \R $ to $\R$ that can be arbitrarily chosen.
The following theorem establishes a wide range of identification strategies.


\begin{thm}[Main]
    \label{thm:ID_main}
    Under Assumptions \textup{\ref{assum:setting}} to \textup{\ref{assum:CTP-A}}, suppose\\
    \mbox{}\hspace{0.5cm}(a) $A^W$ can be explicitly calculated with respect to $\FF$.\\
    \mbox{}\hspace{0.5cm}(b) $\E\left[\int_0^T (H^i_s)^2(\gamma_0) \dx{[M^{W^j}, M^{W^j}]_s}\right] < \infty, \forall i \in [d], j \in [q]$.\\
    Define an $\R^{d\times q}$-valued function
    \begin{equation*}
        \bmu(t;\gamma) \equiv \E\left[\int_0^t H_s(\gamma)\dx{M^{W}_s}\right] \equiv \E\left[\left(\int_0^t H_s^i(\gamma)\dx{M^{W^j}_s}\right)_{i \in [d]; j \in [q]}\right], t \in [0, T],
    \end{equation*}
    where the integral is in the \Ito sense with respect to the filtration $\mathbb{G}$.  
    Arrange elements of $\bmu(t;\gamma)$ into a $dq$-column vector, denoted as $\tbmu(t;\gamma)$. Let $V$ be an arbitrary $dq \times dq$ positive-definite weighting matrix. Then
    \begin{equation}
        g\left(\Phi(m_0), \gamma\right) \equiv \int_0^T\tbmu(t;\gamma)^\intercal V \tbmu(t;\gamma)\dx{t}
        \label{eq:extremum}
    \end{equation}
    is an identification strategy for $\gamma_0$ for all $m_0 \in \mathcal{M}$ in the sense of Definition \textup{\ref{defn:id}}. Equivalently,
    \begin{equation}
        \bmu(t;\gamma)= \mathbf{0}^{d\times q}, \forall t \in [0,T]
        \label{eq:population_ee}
    \end{equation}
    identifies $\gamma_0$.
\end{thm}
\begin{remark} 
    $h_i$ can be any continuous functions from $\R^{p+q+2}$ to $\R$. For example, \[H^i_t(\gamma) = \left[Y_T - \tau_T(\rW_T;\gamma)\right]^i\] using polynomials. It may often be helpful to incorporate covariate information in the moment conditions to increase efficiency, e.g. $H^i_t(\gamma) = Z_{t-}^j\left(Y_T - \tau_T(\rW_T;\gamma)\right)$, where $Z^j$ is an arbitrary component of the covariate process. 
\end{remark}
\begin{remark} 
    If $\Gamma$ has dim$(\Gamma)$ dimensions with dim$(\Gamma) < \infty$, then a finite number of moment conditions may suffice to point-identify $\gamma_0$. They can be formulated by applying Equation \ref{eq:extremum} or \ref{eq:population_ee} to $k$ arbitrary time points $t_1, \ldots, t_k$ such that $kdq \ge \text{dim}(\Gamma)$, e.g.
    \[\bmu(t_i;\gamma)= \mathbf{0}^{d\times q}, i \in [k].\]
\end{remark}
Intuitively, our identification strategy is based on a stochastic integral under filtration $\FG$. When $\gamma = \gamma_0$, the integrand of the integral is a $\FG$-predictable process that depends on the value of the potential outcome $Y_T^{\bar{\boldsymbol{0}}}$ in the future, which is by definition adapted to $\FG$. The integrator $M^W$ is both an $\FF$- and $\FG$-local martingale. Therefore, we can use it as an integrator when it is considered as a $\FG$-local martingale, and identify it from observable data when it is considered as an $\FF$-local martingale. Then, under weak regularity conditions, the $\FG$-stochastic integral is a zero-mean martingale. 

\begin{proof}[Proof of Theorem \ref{thm:ID_main}]
    Without loss of generality, we will prove that $\gamma_0$ is a solution to 
    \begin{equation}
        \bmu(T;\gamma)= \mathbf{0}^{d\times q}
    \end{equation}
     for all $m_0$ in $\mathcal{M}$. It follows the same argument to prove $\bmu(t;\gamma)= \mathbf{0}^{d\times q}$ for $t\ne T$.

    (1) We will first show that $H^i_t(\gamma_0)$ is a $\FG$-predictable process.

    By Assumption \ref{assum:CTC} and \ref{assum:CSM}, $Y_T^{\bar{\boldsymbol{0}}} = Y_T - \tau_T(\rW_T;\gamma_0)$.
    Define $\mathscr{P}$ to be the $\FG$-predictable $\sigma$-algebra on $\Omega \times \mathcal{T}$, which is generated by all continuous $\FG$-adapted processes. Then, obviously the constant process $Y_T^{\bar{\boldsymbol{0}}}$ is a $\FG$-predictable process. Since by Assumption \ref{assum:setting}, $W_{t}, X_{t}$ are $\FG$-adapted \cadlag processes, then their left-limit processes are left-continuous (i.e. c\`ag) processes, and thus are $\FG$-predictable processes. The deterministic function $I_t = t$ is clearly also $\FG$-predictable. Thus, since each $h^i$ is a Borel-measurable function, then 
    $h^i(Y_T^{\bar{\boldsymbol{0}}}(\omega), W_{t-}(\omega), X_{t-}(\omega), t)$ is also $\mathscr{P}$-measurable, and therefore is a $\FG$-predictable process.

    (2) We then show that $\int_0^T H_s^i(\gamma_0)\dx{M^{W^j}_s}$ with respect to filtration $\FG$ is a zero-mean martingale under $m_0$.

    By Assumption \ref{assum:setting}, \ref{assum:NID}, and \ref{assum:CTP-A}, $M^{W^j}$ is a nontrivial $\FF$- and $\FG$-local martingale with $M^{W^j}_0 = 0$, and it can be explicitly identified from observable data. Taking $H^i_t(\gamma_0)$ as the integrand and $M^{W^j}$ as the integrator,
    since $\E\left[\int_0^T (H^i_s)^2(\gamma_0) \dx{[M^{W^j}, M^{W^j}]_s}\right] < \infty$,
    by \Ito Isometry, the resulting stochastic integral is an $L^2$-bounded $\mathbb{G}$-martingale starting from zero. Thus, we have proved Equation \eqref{eq:population_ee}, and it can be easily extended to Equation \eqref{eq:extremum}.
\end{proof}

By the uniqueness of canonical decompositions, there is a function mapping $W$ to its compensator $A^W$, which establishes the existence of a nonparametric point-identification strategy for $A^W$. Thus, our identification is \emph{nonparametric} in nature, since $\Gamma$ can also be infinite-dimensional.
We additionally require that $A^W$ can be explicitly calculated. This is the case for any discrete-time adapted processes embedded in continuous time by definition (thus including Example 1), many counting processes \citep{aven1985theorem}, many processes specified by stochastic differential equations, and special semimartingales that satisfy certain regularity conditions, e.g. \citet{knight1991calculating}.

We show below applications of Theorem \ref{thm:ID_main} to Examples 1 and 3. See Section \ref{sec:simu} for a detailed application of Theorem \ref{thm:ID_main} to estimate causal parameters in Example 2.
\paragraph{Example 1 (Continued)}
By Equation \eqref{eq:eg1_csm} and Assumption \ref{assum:CTC}, 
\[ Y_T^{\bbzero} = Y_T - \tau_T(\rW_T;\eta_1) = Y_T - \eta_1\int_0^T W_s \dx{s}. \]
In the discrete-time representation, the $\FF'$-compensator of $W'$ can be explicitly written as $A^{W'}_k = \sum_{i=1}^k \E[W'_i - W'_{i-1}|\fF'_{i-1}]$, whose continuous-time representation $A^W_t = \sum_{i=1}^{\floor*{Jt/T}} \E[W'_i - W'_{i-1}|\fF'_{i-1}]$ is the $\FF$-compensator of $W$. Since we have verified that the DGP in Example 1 satisfies Assumption \ref{assum:setting} to \ref{assum:CTP-A}, thus we can apply Theorem \ref{thm:ID_main} with an arbitrary function $h$ in $\R$ and a weighting matrix $V$. We choose 
\[h\left(Y_T - \eta_1\int_0^T W_s \dx{s}, W_{t-}, X_{t-}, t\right) = Y_{t-}\left(Y_T - \eta_1\int_0^T W_s \dx{s}\right)\] 
and $V = 1$. Note that $\gamma = \eta_1$ and dim$(\Gamma) = 1$. Thus we use the fact that $\bmu(T;\eta_1)= {0}$ and get a valid identification strategy
\[g(\eta_1) = \left\{\E\left[\int_0^T Y_{t-}\left(Y_T - \eta_1\int_0^T W_s \dx{s}\right) \dx{(W_t - A^W_t)}\right] \right\}^2. \]
Since the integrator is a step function, after simplification we have the discrete-time representation as
\begin{equation*}
    g(\eta_1) = \left\{\E\left[ \sum_{i=1}^J Y'_{i-1}\left(Y_J^{\prime} - \eta_1\frac{T}{J}\sum_{l = 0}^{J-1}W_l'\right) \left\{W'_i - \E(W'_i|\rW'_{i-1}, \rX'_{i-1})\right\} \right]\right\}^2.
\end{equation*}
The above function from the continuous-time strategy coincides with Equation \eqref{eq:eg1_f} in essence, which is derived under a discrete-time strategy.

\paragraph{Example 3 (Continued)}
By Equation \eqref{eq:eg3_csm} and Assumption \ref{assum:CTC}, 
\[ Y_T^{\bbzero}= Y_T - \tau_T(\rW_T; \alpha_0, \alpha_2) = Y_T + \int_0^T W_{s-}\dx{U^2_s(\alpha_0, \alpha_2)},\]
where $\gamma = (\alpha_0, \alpha_2)$ is the causal parameter.
In the proof of Proposition \ref{prop:eg3}, we have shown that the $\FF$-local martingale part of $W$ is 
\[ M^W_t =  \int_0^t U^1_{s-}(\alpha_0, \alpha_1)(1 - W_{s-})\dx{({N}_s} - s),\]
where $\alpha_1$ is a nuisance parameter. By choosing 
\[h^i\left(Y_T^{\bbzero}, W_{t-}, X_{t-}, t\right) = (Z_{t-})^iY_T^{\bbzero}, i = 1,2,3,\]
Theorem \ref{thm:ID_main} gives
\[\E\left[\int_0^T(Z_{s-})^i\left(Y_T + \int_0^T W_{s-}\dx{U^2_s(\alpha_0, \alpha_2)}\right)U^1_{s-}(\alpha_0, \alpha_1)(1 - W_{s-})\dx{({N}_s} - s)\right] = 0, i = 1,2,3.\]
Solving the above equations will identify causal parameters $\alpha_0,\alpha_2$. Equivalently, if we define $T_W$ to be the jump time of $W$, which is observable, and let $T_W = \infty$ if $W$ never jumps in $\mathcal{T}$, and define $a\wedge b \equiv \min(a,b)$, the above equations can be simplified as
\[\E\left[\left\{ Y_T + W_{\iota^2(\alpha_0, \alpha_2)-}\indicator{\iota^2(\alpha_0, \alpha_2)\le T} \right\}\left\{(Z_{T_W})^i \indicator{T_W\le T} - \int_{\iota^1(\alpha_0, \alpha_1)\wedge T}^{T_W\wedge T}(Z_t)^i\dx{t}  \right\}\right] = 0.\]

\subsection{An inverse weighting aspect of Theorem \ref{thm:ID_main}}
\label{subsec:weight}
Weighting methods are widely used in causal inference, e.g., Inverse Probability-of-Treatment Weighting (IPTW) in conjunction with marginal structural models \citep{robins2000marginal}, and balancing weights \citep{ben2021balancing}. We show that after straightforward manipulations of the integrator $\dx{M^W}$ in Theorem \ref{thm:ID_main}, we will have a weighting strategy. Intuitively, by using the \emph{Inverse Rate of change of the Treatment} $W$'s $\FF$-predictable part ${1}/{{a^{W^j}_s}}$ as weights (IRTW), we create a pseudo-population where subjects with different treatments are ``exchangeable". 

We first state a new ``Positivity''-type assumption below. Special cases of it considering counting process treatments have appeared in existing continuous-time IPTW literature, e.g. \citet{hu2021joint}. This alternative ``Positivity'' assumption will replace Assumption \ref{assum:CTP-A} in the weighting identification strategy. Recall that $W \equiv (W^1, \ldots, W^q)$.


\begin{assum}[Continuous-Time Positivity B (CTP-B)] 
    The $\FF$-compensator of $W$ admits the form $A^W_t = \int_0^t a^W_s \dx{s}$, where $a^W_s \equiv (a^{W^1}_s, \ldots, a^{W^q}_s)$ is an $\FF$-predictable process, and $a^{W^i}_s \ne 0, \forall i \in [q], \forall s \in \mathcal{T}$.
    \label{assum:CTP-B}
\end{assum}
CTP-B implies that we can find a well-behaved rate of change function of $a^{W^j}$. For instance, the intensity process of a counting process treatment could be positive for all $t$. One concrete example is that 
\[\dx{W_t} = \lambda_t\exp{(Z_{t-} + Y_{t-})}\dx{N_t},\] 
where $\lambda_t$ is a deterministic baseline function that is positive for all $t$, and $N_t$ is a standard Poisson driving process. Then, $a^W_t =\lambda_t\exp{(Z_{t-} + Y_{t-})} > 0$.
Note that CTP-B can be more restrictive than its alternative CTP-A (Assumption \ref{assum:CTP-A}). 
However, when it is satisfied, it allows an alternative identification strategy. 

We formally state our weighting strategy in Corollary \ref{cor:ID_weighting} below, where ${1}/{{a^{W^j}_s}}$ is used as weights. 
Define $H$ as in Equation \eqref{eq:H}. Define $C_t \equiv c\left(W_{t-}, X_{t-}, t\right)$ and $c$ is a real-valued Borel-measurable function which can be arbitrarily chosen.
For a generic jointly measurable process $\Upsilon$, define its $\FF^{C}$-predictable projection as $^{p}\Upsilon$, which is the unique $\FF^{C}$-predictable process such that ${}^{p}\Upsilon_\tau = \E[\Upsilon_{\tau}|\fF^{C}_{\tau-}]$ for all $\FF^{C}$-predictable times $\tau$.

\begin{cor}[Weighting]
    Under Assumptions \textup{\ref{assum:setting}}, \textup{\ref{assum:NID}}, \textup{\ref{assum:CTC}}, \textup{\ref{assum:CSM}}, and \textup{\ref{assum:CTP-B}}, suppose that for all $i \in [d], j \in [q]$,\\
    \mbox{}\hspace{0.5cm}(a) $a^W$ can be explicitly calculated,\\
    \mbox{}\hspace{0.5cm}(b) $\E\left[\int_0^T \left\{\frac{1}{{a^{W^j}_s}}\left(H^i_s(\gamma_0) - {}^{p}H^i_s(\gamma_0)\right)\right\}^2 \dx{[M^{W^j}, M^{W^j}]_s}\right] < \infty$,\\
    \mbox{}\hspace{0.5cm}(c) $U_t^{i,j} \equiv  \int_0^t \frac{1}{{a^{W^j}_s}}\left({H^i_s(\gamma_0) - {}^{p}H^i_s(\gamma_0)}\right)\dx{W^j_s}$ is locally integrable.\\
    Define an $\R^{d\times q}$-valued function
    \begin{equation}
        \bmu(t;\gamma) \equiv \E\left[\left\{\int_0^t \frac{1}{{a^{W^j}_s}}\left({H^i_s(\gamma) - {}^{p}H^i_s(\gamma)}\right)\dx{W^j_s}\right\}_{i \in [d]; j \in [q]}\right], t \in [0, T],
        \label{eq:cor_ee}
    \end{equation}
    where the integral is in the \Ito sense with respect to the filtration $\mathbb{G}$.  
    Arrange elements of $\bmu(t;\gamma)$ into a $dq$-column vector, denoted as $\tbmu(t;\gamma)$. Let $V$ be an arbitrary $dq \times dq$ positive-definite weighting matrix. Then
    \begin{equation*}
        g\left(\Phi(m_0), \gamma\right) \equiv \int_0^T\tbmu(t;\gamma)^\intercal V \tbmu(t;\gamma)\dx{t}
    \end{equation*}
    is an identification strategy for $\gamma_0$ for all $m_0 \in \mathcal{M}$ in the sense of Definition \textup{\ref{defn:id}}. Equivalently,
    \begin{equation*}
        \bmu(t;\gamma)= \mathbf{0}^{d\times q}, \forall t \in [0,T]
    \end{equation*}
    identifies $\gamma_0$.
    \label{cor:ID_weighting}
\end{cor}

\begin{remark}
    In the simplest case, by setting $c \equiv 0$, ${}^{p}H^i_s(\gamma) = \E[H^i_s(\gamma)]$. More complex $c$'s may increase efficiency in estimations.
\end{remark}
Intuitively, $\E\left[\int_0^t \left({H^i_s(\gamma) - {}^{p}H^i_s(\gamma)}\right)\dx{W^j_s}\right]$ may not be zero due to the correlation between the treatment increments $\dx{W^j_s}$ and the centered potential outcomes ${H^i_s(\gamma) - {}^{p}H^i_s(\gamma)}$. This correlation is removed by reweighting with the ``propensity of treatment" $1/{a^{W^j}_s}$. See Appendix A for the formal proof.

\section{Simulation study}
\label{sec:simu}

To provide a concrete application of the main identification strategy in Theorem \ref{thm:ID_main}, we demonstrate how to estimate counterfactual means from data generated by the DGP in Example 2. The treatment plan of interest is $w_s \equiv 1, \forall s$, denoted as $\bbone$, and the estimand is $\E[Y_T^{\bbone}]$. Note that we focus on the identification in this paper, and thus implement a working estimation procedure below. We will study the theoretical properties of the estimators in future work.

\begin{figure}
    \centering
    \includegraphics[width = 1\textwidth]{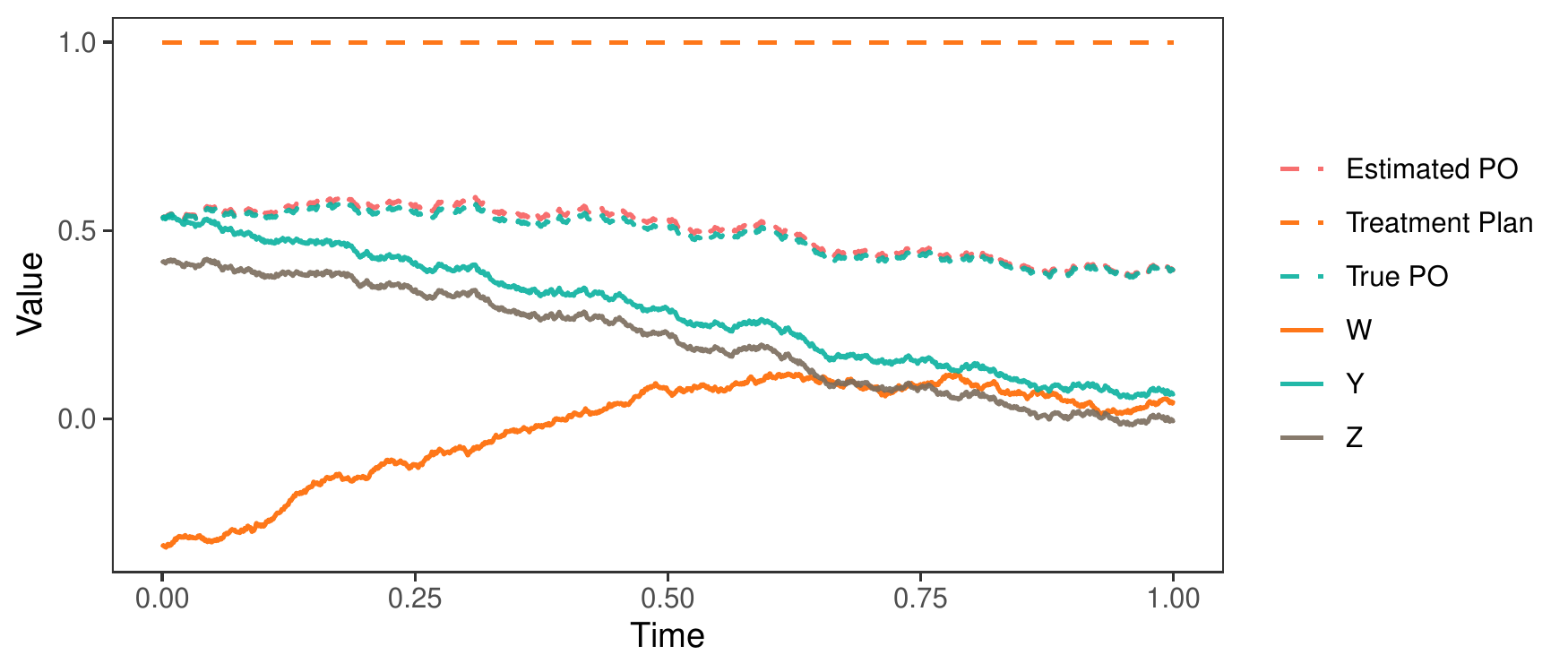}
    \caption{Trajectories of the outcome process $Y$, the treatment process $W$, the covariate process $Z$, the treatment plan $\rw \equiv \bbone$, the true potential outcome (PO) process $Y^{\bbone}$, and the estimated PO process $\hat{Y}^{\bbone}$. }
    \label{fig:simu}
\end{figure}

We generate $n = 2000$ full trajectories\footnote{Numerically, we use the Euler scheme \citep{iacus2009simulation} to simulate trajectories}
according to Equation \eqref{eq:3dOU} on $\mathcal{T} = [0, 1]$ with 
\[\boldsymbol{\beta}^* =\begin{pmatrix}
    &1, &-0.5 , &0.2 \\
    &-0.7 , &1, &-0.6 \\
    &0.3 , &0.2 , &1
\end{pmatrix}, \boldsymbol{\sigma}^*= \begin{pmatrix}
    &0.1, &0\\
    &0, &0.1 \\
    &0.1, &0
\end{pmatrix},\begin{pmatrix} Y_0\\W_0\\Z_0 \end{pmatrix} \sim \text{MVN}(\mathbf{0}, I), \]
where MVN is the multivariate normal distribution. A realization is shown in Figure \ref{fig:simu}.
By Equation \eqref{eq:eg2_csm}, we correctly specify the individual treatment effect $\tau_T(\rw)$ as 
\[\tau_T(\rw;\boldsymbol{\gamma}) = \int_0^T \gamma_1 w_t \left\{ \gamma_2 e^{\gamma_3(t-T)} + (1 - \gamma_2)e^{\gamma_4(t-T)} \right\}\dx{t},\]
where the causal parameter $\boldsymbol{\gamma}$ is a function of $\boldsymbol{\beta}$, and the true value $\boldsymbol{\gamma}^* = (0.50,0.092,0.76,1.2)$. 
Since $\text{dim}(\Gamma) = 4$, we need to specify at least 4 estimating equations. Define 
\begin{equation*}
    \begin{aligned}
        &H^1_t(\boldsymbol{\gamma}) = Y_T - \tau_T(\rW_T;\boldsymbol{\gamma}),
        &H^2_t(\boldsymbol{\gamma}) = Z_{t-}\{Y_T - \tau_T(\rW_T;\boldsymbol{\gamma})\}\\
        &H^3_t(\boldsymbol{\gamma}) = W_{t-}\{Y_T - \tau_T(\rW_T;\boldsymbol{\gamma})\}, 
        &H^4_t(\boldsymbol{\gamma}) = Y_{t-}\{Y_T - \tau_T(\rW_T;\boldsymbol{\gamma})\}\\
    \end{aligned}
\end{equation*}
To apply Theorem \ref{thm:ID_main}, we need to explicitly compute $M^W$. Thus, we define the following correctly specified nuisance model
\[\dx{W_t} = -(\nu_{1}Y_t + \nu_{2}W_t + \nu_{3}Z_t)\dx{t} + \nu_{4} \dx{B}_{t}\]
where $\boldsymbol{\nu} = (\nu_{1}, \nu_{2}, \nu_{3}, \nu_{4})$ is the nuisance parameter whose true value is $\boldsymbol{\nu}^* = (-0.7, 1, -0.6, 0.1)$, and $B$ is a standard Brownian motion. Denote $M^W_s(\boldsymbol{\nu}) = W_s - W_0 + \int_0^s(\nu_{1}Y_t + \nu_{2}W_t + \nu_{3}Z_t)\dx{t}$, and let $0 = t_0 < \ldots < t_\ell = 1$ be a time grid on $\mathcal{T}$ with $\ell\ge 2$. Then $\boldsymbol{\nu}$ can be easily estimated by solving the following estimating equations:
\begin{equation}
\Prob_n[M^W_{t_i}(\boldsymbol{\nu})] = 0; \Prob_n[M^W_{t_i}(\boldsymbol{\nu})]^2 = \nu_4^2t_i; \Prob_n\left[M^W_{t_{i-1}}(\boldsymbol{\nu})\left(M^W_{t_{i}}(\boldsymbol{\nu})-M^W_{t_{i-1}}(\boldsymbol{\nu})\right)\right] = 0, 
\label{eq:eg2_ee_nu}
\end{equation}
$i \in [\ell] $, where $\Prob_n$ is the empirical measure.
Denote the estimate from Equation \eqref{eq:eg2_ee_nu} as $\hat{\boldsymbol{\nu}}$, then $M^{W}_t$ can be explicitly estimated as 
\[\hat{M}^{W}_t = W_t - W_0 + \int_0^t(\hat{\nu}_{1}Y_s + \hat{\nu}_{2}W_s + \hat{\nu}_{3}Z_s)\dx{s}.\]
By Equation \eqref{eq:extremum} and choosing $V$ to be the identity matrix, we minimize the sample criteria function
\begin{equation}
g_n(\boldsymbol{\gamma}) = \sum_{i = 1}^4\left[\Prob_n\left( \int_0^T H_s^i(\boldsymbol{\gamma})\dx{\hat{M}^{W}_s} \right)\right]^2,
\label{eq:eg2_ee_gam}
\end{equation}
and define our estimator $\hat{\boldsymbol{\gamma}} = \arg\min_{\boldsymbol{\gamma}\in\Gamma}g_n(\boldsymbol{\gamma})$. 

In fact, we can jointly solve Equation \ref{eq:eg2_ee_nu} and \ref{eq:eg2_ee_gam}, and in our simulation, $\hat{\boldsymbol{\gamma}} = (0.59, 0.074, 0.75, 1.3)$,  
and $\hat{\boldsymbol{\nu}} = (-0.65, 0.96, -0.51, 0.10)$. 
Then an estimate of $\E[Y_T^{\bbone}]$ is 
\[\hat{\E}[Y_T^{\bbone}] = \Prob_n\left[Y_{T} - \tau_T(\rW_{T}; \hat{\boldsymbol{\gamma}}) + \tau_T(\bbone; \hat{\boldsymbol{\gamma}})\right] = 0.329,\]
while the true value is 0.292. In Figure \ref{fig:simu}, we visualize the observable outcome $Y_t$, counterfactual outcome $Y^{\bbone}_t$, and the estimated counterfactual outcome $\hat{Y}^{\bbone}_t$.

\section{Discussion}

In this paper, we propose novel causal identification strategies for general continuous-time \cadlag processes under time-varying treatments and treatment-confounder feedbacks, which has been inspected from two perspectives: \emph{orthogonalization} and \emph{weighting}. We focus on the special semimartingale treatments, which include the rarely studied but widely seen case of continuously valued treatments. 

The proposed strategies rely on an ``unconfoundedness''-type assumption, the \emph{No Information Drift} assumption (Assumption \ref{assum:NID}). This key assumption indicates an \emph{invariance under a change of filtration} (the factual filtration $\FF$ versus the counterfactual filtration $\FG$). It has an intuitive interpretation that the treatment decision maker does not possess information about future potential outcomes that is unmeasured. This condition is usually unverifiable in observational studies. Thus, its validity should be carefully deduced and discussed based on domain knowledge, and the sensitivity of the results should be checked.

Similar to the discrete-time Marginal Structural Models with IPTW, the proposed strategies need two models. First, a causal structural model (CSM) linking potential outcomes to deterministic treatment plans (Assumption \ref{assum:CSM}) which specifies individual treatment effects. 
Second, a model of the observable treatment process based on the observable history $\FF$. Note that the $\FF$-model of $W$ is a nuisance model and does not need to be a causal one, and that no models for the covariate processes $Z$ are needed in this identification framework.

For different choices of $H$ or $C$, the strategies in Theorem \ref{thm:ID_main} and Corollary \ref{cor:ID_weighting} may point or partially identify the true causal parameter $\gamma_0$. For the subsequent estimation, statistical inference methodology is readily available in this case as a result of the recently fast-growing literature in econometrics for inference on possibly partially identified models through moment conditions. See \citet{molinari2020microeconometrics, canay2017practical, bontemps2017set} for recent advances. Therefore, we avoid making unnecessary assumptions to ensure point-identification for inference as in most traditional M-/Z-estimation literature, e.g. \citet[Chapter.~5]{van2000asymptotic}.

Many open questions remain for future work: (1) The choice of $H$ and $C$ may impact the efficiency of the estimation. Thus, knowing how to choose an ``optimal'' estimating function is of practical value. (2) Dynamic treatment regimes are common and often of interest in practice. We plan to extend our framework to this setting. (3) Throughout we focus on the problem of causal identification with fully observed trajectories. Our ongoing work is dedicated to solving the more complex causal problem of continuous-time DGPs under discrete-time observation patterns, where the current work provides a theoretical basis.

\textbf{Acknowledgements}: We are grateful to 
Fan Li 
and 
P. M. Aronow 
for helpful comments and suggestions. This work was supported by NIH grant NICHD 1DP2HD091799-01.

\appendix

\section{Proofs}

\subsection{Proof of Proposition \ref{prop:eg1}}

\begin{proof}
    In this proof, we will consider a more general case for an arbitrary discrete-time treatment plan $\rw'_J$, thus including the special case of $\rw'_J = \bar{\boldsymbol{0}}'$. 

    Proof of (i):\\
    Using the above notation, $M^W$ is an $\mathbb{F}$-local martingale, and we will show $M^W$ is also a $\mathbb{G}$-local martingale. 
    Define the discrete-time representation of decomposition of $W$ as $W'_t = W'_0 + M_t^{W'} + A_t^{W'}$. Then, it suffices to prove $M_t^{W'}$ is a $\mathbb{G}'$-local martingale.
    Since $M^{W'}$ is an $\mathbb{F'}$-local martingale, there exists a localizing sequence of $\mathbb{F'}$-stopping times $\tau_i\uparrow J, i = 1, 2\ldots$, taking integer values such that each stopped process $(M^{W'})^{\tau_i}$ is an $\mathbb{F'}$-martingale. 

    To show that Assumption \ref{assum:DTSE} implies Assumption \ref{assum:NID}, it would suffice if we can show $\E[(M^{W'})^{\tau_i}_k|\fG'_{k-1}] = (M^{W'})^{\tau_i}_{k-1}, \forall i$. Since
    \begin{equation}
        \begin{aligned}
            \E[M^{W'}_k|\fG'_{k-1}] &= \E[{W'_k} - {W'_0} - A^{W'}_k|\fG'_{k-1}]\\
            & \text{ (since ${W'_0}$ is $\fG'_{k-1}$-measurable.)} \\
            &= \E[{W'_k}|\fG'_{k-1}] - \E[A^{W'}_k|\fG'_{k-1}] - {W'_0} \\
            &\text{ (by the definition of $\fG'_{k-1}$, and that $A^{W'}_k$ is $\FF'$-predictable thus $A^{W'}_k$ is $\fF'_{k-1}$-measurable. )}  \\
            &= \E[{W'_k}|\fF'_{k-1} \bigvee \sigma(\rY^{\prime\rw'}_J)] - \E[A^{W'}_k|\fF'_{k-1}]  - {W'_0} \\
            &\text{ (by Assumption \ref{assum:DTSE}, and tower property of conditional expectation.)}  \\
            &= \E[{W'_k}|\fF'_{k-1}] - \E[A^{W'}_k|\fF'_{k-1}]  - {W'_0} \\
            &= \E[{W'_k} - A^{W'}_k - {W'_0}|\fF'_{k-1}] \\
            &= \E[M^{W'}_k|\fF'_{k-1}]
        \end{aligned}
        \label{eq:dtct1}
    \end{equation}
    Then for a stopping time $\tau$ in $(\tau_i)$, 
    \begin{equation*}
        \begin{aligned}
            \E[(M^{W'})^{\tau}_k|\fG'_{k-1}] &= \E[(M^{W'}_{{\tau}\wedge k})|\fG'_{k-1}] \\
            &= \E[M^{W'}_{\tau}|\fG'_{k-1}]\indicator{\tau \le k-1} + \E[M^{W'}_{k}|\fG'_{k-1}]\indicator{\tau \ge k} \\
            & \text{ (by Equation \eqref{eq:dtct1}.)} \\
            &= \E[M^{W'}_{\tau \wedge (k-1)}|\fG'_{k-1}]\indicator{\tau \le k-1} + \E[M^{W'}_{k}|\fF'_{k-1}]\indicator{\tau \ge k} \\
            & \text{ (since $\tau$ is an $\FF'$-stopping time, and $M^{W'}$ is $\FF$-optional,}\\
            & \text{ $M^{W'}_{\tau \wedge (k-1)}$ is $\fF'_{\tau \wedge (k-1)}$-measurable, and thus $\fF'_{k-1}$- and $\fG'_{k-1}$-measurable)} \\
            &= \E[M^{W'}_{\tau \wedge (k-1)}|\fF'_{k-1}]\indicator{\tau \le k-1} + \E[M^{W'}_{k}|\fF'_{k-1}]\indicator{\tau \ge k} \\
            &= \E[(M^{W'})^{\tau}_k|\fF'_{k-1}] \\
            & \text{ (since $M^{W'}$ is $\FF$-local martingale.)} \\
            &= (M^{W'})^{\tau}_{k-1}
        \end{aligned}
    \end{equation*}
    Thus, we have shown that $M^{W'}$ is also a $\mathbb{G}'$-local martingale. Hence, we have finished the proof.

    Proof of (ii):\\
    Since $M^W$ is both an $\FF$- and $\FG$-martingale, then under the discrete-time representation, $\E[M^{W'}_k|\fG'_{k-1}] = \E[M^{W'}_k|\fF'_{k-1}]$. By similar arguments in Equation \eqref{eq:dtct1}, we have 
    \begin{equation*}
        \begin{aligned}
            & \E[{W'_k}|\fF'_{k-1} \bigvee \sigma(\rY^{\prime\rw'}_J)] - \E[A^{W'}_k|\fF'_{k-1}]  - {W'_0}\\
            &= \E[M^{W'}_k|\fG'_{k-1}] \\
            &= \E[M^{W'}_k|\fF'_{k-1}] \\
            &= \E[{W'_k}|\fF'_{k-1}] - \E[A^{W'}_k|\fF'_{k-1}]  - {W'_0}.
        \end{aligned}
    \end{equation*}
    Thus, $ \E[{W'_k}|\fF'_{k-1}] = \E[{W'_k}|\fF'_{k-1} \bigvee \sigma(\rY^{\prime\rw'}_J)]$.
\end{proof}

\subsection{Proof of Proposition \ref{prop:eg2}}

\begin{proof}
    We prove for a general treatment plan $\rw(T)$.
    By the DGP, 
    \[\dx{W_t} = -(\beta_{21}Y_t + \beta_{22}W_t + \beta_{23}Z_t)\dx{t} + \sigma_{22} \dx{B^2}_{t} \]
    By the martingale property of an \Ito diffusion, the decomposition of $W_t$ under $\mathbb{F}$ has the martingale part $M^W_t = \sigma_{22} {B^2_t}$. Then, for all $s < t \le T,$
    \begin{equation*}
        \begin{aligned}
            \E[{B^2_t}|\fG_s] &= \E[{B^2_t}|\fF_s \bigvee \sigma(\rY_T^{\rw(T)})] \\
            & \text{(by the property of the strong solution of a diffusion, $\fF^{Y^{\rw}}_t \subseteq \mathcal{N} \bigvee \sigma(Y_0, Z_0)\bigvee \fF^{B^1}_t$}\\
            & \text{and $\fF_s \subseteq \mathcal{N} \bigvee \sigma(W_0, Y_0, Z_0)\bigvee \fF^{B^1, B^2}_t$. Then by the tower property,)} \\
            &= \E\left[\E\left[{B^2_t} \bigg|\mathcal{N} \bigvee \sigma(W_0, Z_0, Y_0) \bigvee \fF^{B^1}_T \bigvee  \fF^{B^2}_s \right] \bigg| \fF_s \bigvee \sigma(\rY_T^{\rw(T)})\right]\\
            & \text{(by the joint independence of $B^{2}$, $B^1$, and $W_0, Z_0, Y_0$.)}\\
            &= \E\left[\E\left[{B^2_t} \bigg|\fF^{B^2}_s \right] \bigg| \fF_s \bigvee \sigma(\rY_T^{\rw(T)})\right]\\
            & \text{(since $B^2$ is a Brownian motion with respect to its natural filtration.)}\\
            &= \E\left[B^2_{s}\bigg|\fF_s \bigvee \sigma(\rY_T^{\rw(T)})\right]\\
            & \text{(since $B^2$ is $\FF$-martingale, $B^2_{s} \in \fF_s$.)} \\
            &= B^2_s
        \end{aligned}
    \end{equation*}
    Hence $M^W_t$ is also a $\mathbb{G}$-martingale, and we finished the proof.
\end{proof}
    
\subsection{Proof of Proposition \ref{prop:eg3}}

\begin{proof}
    (I) We first compute the $\FF$-decomposition of $W$. 
    
    Denote $\tilde{N}_t = N_t - t$. By the property of Lebesgue-Stieltjes integral,
    \begin{equation}
        W_t = \int_0^t U^1_{s-}(1 - W_{s-})\dx{s} + \int_0^t U^1_{s-}(1 - W_{s-})\dx{\tilde{N}_s}.
        \label{eq:eg3_W}
    \end{equation}
    Define $K_t \equiv \int_0^t U^1_{s-}(1 - W_{s-})\dx{\tilde{N}_s}$.\\ 
    (i) We prove that $K_t$ is an $\FF$-martingale. Since $U^1, W$ are \cadlag and $\FF$-adapted, and $\FF$ is a subfiltration of $\FH$, thus the integrand $U^1_{s-}(1 - W_{s-})$ is an $\FH$-predictable bounded process. Since $\tilde{N}_s$ is an $\FH$-martingale, and $\E\left[\int_0^T(U^1_{s-}(1 - W_{s-}))^2 \dx{[\tilde{N}, \tilde{N}]_s}\right] < \infty$, thus, the stochastic integral is an $\FH$-martingale. Additionally, since $K_t = W_t - \int_0^t U^1_{s-}(1 - W_{s-})\dx{s}$, and both terms on the right hand side are $\FF$-adapted, thus $K_t$ is $\FF$-adapted. Therefore, $K_t$ is an $\FF$-martingale. \\
    (ii) Since $\int_0^t U^1_{s-}(1 - W_{s-})\dx{s}$ is an $\FF$-adapted continuous increasing process, it is an $\FF$-predictable finite variation process.
    By (i) and (ii), Equation \eqref{eq:eg3_W} is the $\FF$-canonical decomposition of $W$.

    (II) Next, we show that $\FG$-canonical decomposition of $W$ is also Equation \eqref{eq:eg3_W}. This is equivalent to proving that $K_t$ is also a $\FG$-martingale. We know that $\fG_t = \fF_t \bigvee \sigma(\rY_T^{\bar{\boldsymbol{0}}}) = \fF_t \bigvee \sigma(\overbar{U}_T^2) = \fF_t \bigvee \sigma(\iota^2)$. Define $\mathbb{L} = (\mathscr{L}_t)$, with $\mathscr{L}_t = \fH_t \bigvee \sigma(\iota^2)$. Thus, $\fF_t \subseteq \fG_t \subseteq \mathscr{L}_t$.
    Since $U^1_{t-}(1 - W_{t-})$ is $\FF$-predictable and bounded, it is thus also $\mathbb{L}$-predictable and bounded. By definition, we have 
    \[\E[\tilde{N}_t| \mathscr{L}_s] = \E[\tilde{N}_t| \fH_s \bigvee \sigma(\iota^2)].\]
    Since $\iota^2$ is an $\nFF^B$-stopping time, it is $\fF^B_T\bigvee \mathcal{N}$-measurable. By the independence of $B$ and $N$, we have $\E[\tilde{N}_t| \fH_s \bigvee \sigma(\iota^2)] = \E[\tilde{N}_t| \fH_s] = \tilde{N}_s$. Thus, $\tilde{N}$ is an $\mathbb{L}$-martingale. Therefore, $K_t$ is an  $\mathbb{L}$-martingale. Since we know that $K_t$ is $\FG$-adapted, thus, $K_t$ is a $\FG$-martingale. Hence we finished the proof.
\end{proof}

\subsection{Proof of Corollary \ref{cor:ID_weighting}}

\begin{proof}
    In general, our proof is based on the canonical decomposition of the stochastic integral in Equation \eqref{eq:cor_ee}, where the local martingale part is a zero-mean martingale, and the expectation of the compensator is an integral of a constant zero process. 

    In particular, we prove that $\gamma_0$ is a solution to \[\E\left[\int_0^t \frac{1}{{a^{W^j}_s}}\left({H^i_s(\gamma) - {}^{p}H^i_s(\gamma)}\right)\dx{W^j_s}\right] = 0\] for all $i,j,t$.

    (1) Note that by Assumption \ref{assum:CTP-B}, $a^{W^j}$ is $\FF$-predictable and nonzero, thus 
    $\frac{1}{{a^{W^j}_s}}$ is well-defined, $\FF$-predictable, and thus identified by observable data. \\
    (2) By Assumption \ref{assum:setting}, $W_{t-}, X_{t-}$ are $\FF$-predictable, since $c$ is Borel-measurable, $C$ is $\FF$-predictable, and thus ${}^{p}H^i_s(\gamma)$ is $\FF$-predictable for any fixed $\gamma$.\\
    (3) Note that since $U_t^{i,j}$ is locally integrable, thus we have the decomposition
    \begin{equation*}
        \begin{aligned}
            &\E\left[\int_0^t \frac{1}{{a^{W^j}_s}}\left({H^i_s(\gamma_0) - {}^{p}H^i_s(\gamma_0)}\right)\dx{W^j_s} \right] \\
            &= \E\left[\int_0^t \frac{1}{{a^{W^j}_s}}\left({H^i_s(\gamma_0) - {}^{p}H^i_s(\gamma_0)}\right)\dx{A^{W^j}_s} \right] + \E\left[\int_0^t \frac{1}{{a^{W^j}_s}}\left({H^i_s(\gamma_0) - {}^{p}H^i_s(\gamma_0)}\right)\dx{M^{W^j}_s} \right]
        \end{aligned}
    \end{equation*}
    (4) Given Assumption \ref{assum:NID}, \ref{assum:CTC}, and \ref{assum:CSM} and condition (b), using a similar argument in Proof of Theorem \ref{thm:ID_main} and considering (1)(2), the second term on the right hand side is zero.\\
    (5) By Assumption \ref{assum:CTP-B}, we have
    \begin{equation*}
        \begin{aligned}
            &\E\left[\int_0^t \frac{1}{{a^{W^j}_s}}\left({H^i_s(\gamma_0) - {}^{p}H^i_s(\gamma_0)}\right)\dx{A^{W^j}_s} \right]\\
            &= \E\left[\int_0^t \frac{1}{{a^{W^j}_s}}\left({H^i_s(\gamma_0) - {}^{p}H^i_s(\gamma_0)}\right){a^{W^j}_s}\dx{s} \right]\\
            &= \int_0^t \E\left[{H^i_s(\gamma_0) - {}^{p}H^i_s(\gamma_0)}\right]\dx{s} \\
            &= 0
        \end{aligned}
    \end{equation*}
    Hence we have finished the proof.
\end{proof}



\bibliography{fda}

\end{document}